\documentclass[a4paper,reqno]{amsart}
\pdfoutput=1
\usepackage{changes}
\usepackage{amscd,amsthm,amsmath,amssymb,mathtools}
\usepackage{upgreek,bm}
\usepackage{verbatim}
\usepackage{color}
\usepackage{hyperref}
\usepackage{url}
\usepackage{enumerate,enumitem}
\usepackage{graphicx}
\usepackage{epstopdf,epsfig,subfigure}
\usepackage{curve2e}
\usepackage{array}

\usepackage{algorithm}
\usepackage{algorithmic}

\usepackage[numbers,sort&compress]{natbib}

\usepackage{setspace}
 \usepackage[displaymath,mathlines,pagewise]{lineno}

 \let\oldequation\equation
 \let\oldendequation\endequation

 \renewenvironment{equation}
   {\linenomath\oldequation}
   {\oldendequation\linenomath}

 \let\oldalign\align
 \let\oldendalign\endalign

 \renewenvironment{align}
   {\linenomath\oldalign}
   {\oldendalign\linenomath}

\theoremstyle{plain}

\newtheorem{theorem}{Theorem}[section]

\newtheorem{lemma}[theorem]{Lemma}
\newtheorem{corollary}[theorem]{Corollary}

\newtheorem{assumption}[theorem]{Assumption}

\theoremstyle{definition}

\newtheorem{remark}[theorem]{Remark}

\numberwithin{equation}{section}

\usepackage{geometry}
 \usepackage[us]{datetime}

\newcommand{\rest}{\left.\kern-2\nulldelimiterspace\right|_}
\newcommand{\norm}[2]{\left|#1\right|_{#2}}
\newcommand{\dnorm}[2]{\left\|#1\right\|_{#2}}

\newcommand{\zero}{{\mathbf0}}
\newcommand{\Id}{{\mathbf1}}

\newcommand*{\Bigcdot}{\raisebox{-.25ex}{\scalebox{1.25}{$\cdot$}}}

\newcommand{\clA}{{\mathcal A}}
\newcommand{\clB}{{\mathcal B}}
\newcommand{\clC}{{\mathcal C}}
\newcommand{\clD}{{\mathcal D}}
\newcommand{\clE}{{\mathcal E}}

\newcommand{\clI}{{\mathcal I}}
\newcommand{\clJ}{{\mathcal J}}

\newcommand{\clL}{{\mathcal L}}

\newcommand{\clS}{{\mathcal S}}

\newcommand{\clX}{{\mathcal X}}
\newcommand{\clY}{{\mathcal Y}}

\newcommand{\bbN}{{\mathbb N}}

\newcommand{\bbP}{{\mathbb P}}

\newcommand{\bbR}{{\mathbb R}}

\newcommand{\bfA}{{\mathbf A}}
\newcommand{\bfB}{{\mathbf B}}

\newcommand{\bfK}{{\mathbf K}}

\newcommand{\bfR}{{\mathbf R}}
\newcommand{\bfS}{{\mathbf S}}

\newcommand{\fkC}{{\mathfrak C}}

\newcommand{\fkS}{{\mathfrak S}}
\newcommand{\fkT}{{\mathfrak T}}

\newcommand{\bfb}{{\mathbf b}}

\newcommand{\bff}{{\mathbf f}}

\newcommand{\bfh}{{\mathbf h}}

\newcommand{\bfp}{{\mathbf p}}

\newcommand{\bfx}{{\mathbf x}}
\newcommand{\bfy}{{\mathbf y}}
\newcommand{\bfz}{{\mathbf z}}

\newcommand{\rmd}{{\mathrm d}}
\newcommand{\rme}{{\mathrm e}}

\definecolor{DarkBlue}{rgb}{0,0.08,0.45}
\definecolor{DarkRed}{rgb}{.65,0,0}
\definecolor{applegreen}{rgb}{0.55, 0.71, 0.0}

\newcounter{mymac@matlab}
\newcommand{\matlab}{MATLAB%
   \ifnum\value{mymac@matlab}<1%
   \textregistered%
   \setcounter{mymac@matlab}{1}%
   \fi%
  }

\newcommand{\red}{ \color{red} }

\newcommand{\bfPi}{{\mathbf \Pi}}

\usepackage{todonotes}
\newcommand{\todoautc}[3][]{%
    \ifthenelse{\equal{#1}{}}{\todo[size=\scriptsize]{{\bf#2} #3}{}}{\todo[color=#1,size=\scriptsize]{{\bf#2} #3}{}}%
}

\newcommand{\addref}[1][]{{\red\tt addref}}

\begin{document}
\title{Tracking optimal feedback control under uncertain parameters}
\author{Philipp A.~Guth$^{\tt1}$,  Karl Kunisch$^{\tt 2}$, and S\'ergio S.~Rodrigues$^{\tt1}$}
\thanks{
\vspace{-1em}\newline\noindent
\newline\noindent
$^{\tt1}$ Johann Radon Institute for Computational and Applied Mathematics,
  \"OAW, 
  Altenbergerstrasse~69, 4040~Linz, Austria.\newline\noindent
$^{\tt2}$   Institute of Mathematics and Scientific Computing,    Karl-Franzens University of Graz,	     	
Heinrichstrasse~36, 8010 Graz, Austria, and Johann Radon Institute for Computational and Applied Mathematics,
  \"OAW, 
  Altenbergerstrasse~69, 4040~Linz, Austria.\newline\noindent
{\sc Emails}:
{\small\tt   philipp.guth@ricam.oeaw.ac.at,\quad karl.kunisch@uni-graz.at,\quad\\ \hspace*{3.4em}sergio.rodrigues@ricam.oeaw.ac.at}
 }

\begin{abstract}
Optimal control problems of tracking type for a class of linear systems with  uncertain parameters in the dynamics are investigated. An affine tracking feedback control input is obtained by considering the minimization of an energy-like functional depending on a finite ensemble of training/sample parameters. It is computed from the nonnegative definite solution of an associated differential Riccati equation. Simulations are presented showing the tracking performance of the computed input for trained as well as untrained parameters.
\end{abstract}

\maketitle


\pagestyle{myheadings} \thispagestyle{plain} \markboth{\sc P.A. Guth,  K. Kunisch, and S.S. Rodrigues}
{\sc Tracking optimal feedback control under uncertain parameters}



%
%
%


\section{Introduction}
Tracking problems over a finite time-horizon~$T>0$ for linear autonomous control systems in the form
\begin{align}
\dot{y} &= \clA y+ B u, \qquad y(0) = y_{\circ},\notag
\end{align}
are investigated, with state~$y(t)\in H$, for time~$t\in[0,T]$, and~$\dot y\coloneqq\frac{\rmd}{\rmd t}y$. The state space~$ H$ is a separable Hilbert space,~$\clA$ is the infinitesimal generator of a semigroup~$S(t)_{t\ge0}$, and~$B\colon  U\mapsto H$ is a bounded linear operator. The control space~$ U$ is another separable Hilbert space.
The initial condition~$y_\circ \in  H$ is given and the choice of the control input~$u \in L^2(0,T; U)$ is  at our disposal.

In many situations the dynamics depends on uncertain or unknown parameters. Thus, we address the design of a robust feedback control operator for parameter-dependent systems of the form
\begin{align}
\dot{y}_\sigma &= \clA_\sigma y_\sigma+ B u, \qquad y_\sigma(0) = y_{\circ},\label{eq:psys}
\end{align}
with an uncertain/unknown parameter~$\sigma$ in a given set $\fkS\in\bbR^S$, for some positive integer~$S$. More precisely, we aim at driving  the state~$y_\sigma$  as close as possible to a given target function~$g$. For this purpose, if we knew the exact value of~$\sigma$, we could follow a classical strategy by considering the minimization of energy-like functionals as
\begin{equation}\label{eq:J1}
\clJ_1(y_\sigma,u) = \frac12 \int_0^T \left( \|y_\sigma(t)- g(t)\|^2_{H} +  \|u(t)\|^2_{ U} \right) \mathrm dt + \frac{1}{2} \| x_\sigma(T)- g(T)\|_{H}^2,
\end{equation}
subject to~\eqref{eq:psys}. In this way we would obtain a feedback control input~$u(t)=K (t,y_\sigma(t))$, with the input feedback operator~$K=K_\sigma$ depending on~$\sigma$, arriving at
\begin{align}
\dot{y}_\sigma(t) &= \clA_\sigma y_\sigma(t)+ BK_\sigma (t,y_\sigma(t)), \qquad y_\sigma(0) = y_{\circ}.\notag
\end{align}

If we do not know~$\sigma$, we could try to use a guess (or an estimate) $\overline\sigma$ for it. Applying the feedback corresponding to the guess, we would arrive at
\begin{align}
\dot{y}_\sigma(t) &= \clA_\sigma y_\sigma(t)+ BK_{\overline\sigma} (t,y_\sigma(t))=\clA_{\overline\sigma} y_\sigma(t)+ BK_{\overline\sigma} (t,y_\sigma(t))+(\clA_\sigma-\clA_{\overline\sigma}) y_\sigma(t)\notag
\end{align}
 If our estimate is good enough so that~$\clA_\sigma-\clA_{\overline\sigma}$ is small, then, we can hope that this feedback input will provide good tracking properties.

However, finding a good estimate and subsequently computing the optimal input feedback~$K_{\overline\sigma}$ can be a time consuming task and can be impractical for real time applications. So, we propose to design an input control operator~$K=K_\varSigma$, depending on an \emph{a priori} fixed finite subset~$\varSigma\subset\fkS$, but independent of a particular realization of~$\sigma$.

\subsection{Related literature}
We could not find works, in the literature, on finite time-horizon (i.e., $0<T<+\infty$) tracking optimal feedback control problems for a general target~$g$ under uncertainty. Here we propose and analyze  a feedback input control operator inspired by the strategy in~\cite{GuthKunRod23}, for the case~$g=0$ in the case of infinite time-horizon, $T=+\infty$.

The strategy in~\cite{GuthKunRod23} applies classical optimal control theory for linear systems to an auxiliary extended system depending on an ensemble of sample parameters~$\varSigma$. In the context of tracking objectives
we use the optimal control theory developed, for example, in~\cite{hinze} and~\cite[Ch.~8.3]{sontag}. As we shall recall later, after a change of variables as~$x\coloneqq y-g$, the problem of tracking~$g$, under linear dynamics for~$y$, is reduced to the problem of tracking~$0$, under affine dynamics for~$x$, leading us to the theory in~\cite[Part~IV, Ch.~1, Sect.~7.1]{bensoussan2007representation}.

The addressed problem falls into the larger class of optimization under uncertainty, see,  for example~\cite{azmi2023analysis,guth2022parabolic,kunoth2013analytic,martinez2016robust} treating open-loop optimal control problems or stationary optimization problems. The present work focuses on optimal control problems in feedback form.

We underline that the uncertainty enters the system through the operator~$\clA_\sigma$, thus it does not necessarily enter in an affine manner as~$\clA_\sigma y=\clA y+\eta(\sigma)$. Noise~$\eta(\sigma)$ entering the dynamics in an affine manner is for instance investigated in~\cite[Ch.~3.6]{kwakernaak1969linear} and~\cite[Ch.~III]{fleming2006controlled}.

Controlled systems with uncertainties entering the system operator~$\clA_\sigma$ arise, for instance, in the case of parabolic equations with uncertain diffusion, reaction, or convection coefficients. Another case is that of damped wave-like equations with uncertain damping coefficients.

In the context of stabilization (i.e., $T=+\infty$), examples of research towards feedback controls for parameterized systems include~\cite{yedavalli2014robust}, where robustness criteria for linear systems, and error bounds are obtained for the perturbed system and control matrices under which a Riccati based nominal feedback law remains stable. In~\cite{GuthKunRod23_2,kramer2017feedback} online-offline strategies are proposed to stabilize a parameter-dependent controlled dynamical system. See also~\cite{chittaro2018asymptotic}, where stabilizability is investigated for an ensemble of Bloch equations, and~\cite{ryan2014simultaneous}, where a bilinear stabilizing feedback is constructed for an ensemble of oscillators.

In the context of controllability, at/in a given time~$T$, $0<T<+\infty$, the concept of ensemble controllability (controllability of ensembles of systems;  simultaneous controllability), is discussed in \cite{LAZAR2022265,helmke2014uniform,danhane2022conditions}, \cite[Ch.~5]{lions1988controlabilite}, \cite[ Ch.~11.3]{TucsnakWeiss09}. In~\cite{zuazua2014averaged} the notion of averaged controllabity, is discussed and a Kalman-type rank condition is derived; see also~\cite{coulson2019average}.

 \subsection{Contents and notation}\label{sS:cont_not}
 The manuscript is structured as follows. In Section~\ref{sS:extsystem} we consider an extended system with~$N$ copies of the dynamics corresponding to the parameters in a finite training ensemble~$\varSigma\subset\fkS$ and construct a time-dependent feedback input operator~$\bfK_\varSigma\colon [0,T]\times H^N\to U$ for this extended system in Section~\ref{sS:optim-ext}. Then, in Section~\ref{sS:ext-2-ori} we use~$\bfK_\varSigma$ to construct a feedback control~$K_\varSigma\colon [0,T]\times H\to U$ for the original system. Subsequently, we compare the cost of this later feedback with the optimal one in case we knew the uncertain parameter in Section~\ref{S:optimCost}; see Corollary~\ref{coro:smallrob}. Besides, in Section~\ref{C:comp-traj} we also compare the corresponding trajectories and control inputs; see Corollary~\ref{coro:optim-pair-comp}. Finally, results of numerical experiments are reported in Section~\ref{S:numEx}.

\bigskip
Concerning notation, given real numbers~$r<s$ and separable Banach spaces~$\clX$ and~$\clY$, the space of continuous functions from~$[r,s]$ into~$\clX$ is denoted by~$\clC([r,s];\clX)$ and the Bochner space of strongly measurable square integrable functions from the interval~$(r,s)$ into~$\clX$ is denoted by~$L^2(r,s;\clX)$ and we also denote the subspace~$W(r,s;\clX,\clY)\coloneqq \{v \in L^2(r,s;\clX)\,\mid\,\dot{v} \in L^2(r,s;\clY)\}$.
Since the time horizon~$T>0$ will be fixed throughout this manuscript, to shorten the exposition, sometimes we shall denote
\begin{equation}\label{eq:def-Space_T}
{\clX}_T \coloneqq L^2( 0,T;\clX)\quad\mbox{and}\quad W_T(\clX,\clY)\coloneqq W(0,T; \clX,\clY).
\end{equation}

By~$\clL(\clX,\clY)$ we denote the space of linear continuous mappings from~$\clX$ into~$\clY$, and in case~$\clX=\clY$ we use the shorter~$\clL(\clX)\coloneqq\clL(\clX,\clX)$.


\section{Feedback controls for tracking objectives}\label{S:feedback}
We fix a positive integer~$N$ and a finite ensemble~$\varSigma\coloneqq (\sigma_i)_{i=1}^{N}\subseteq\fkS$. Further, we consider a more general version of~\eqref{eq:J1} as follows; see~\cite[Part~IV, Ch.~1, Eq.~(1.2)]{bensoussan2007representation}. We fix two more separable Hilbert spaces, $Y$ and~$Z$, and two bounded linear operators~$Q\colon  H\to Y$ and~$P\colon  H\to Z$. Then, we look for a control input~$u\in L^2(0,T; U)$, which minimizes
\begin{equation}\label{eq:obj}
\begin{split}
\clJ(\bfy_\varSigma,u) &\coloneqq  \int_0^T \left( \frac12\|u\|_ U^2+\frac{1}{2N}  \sum_{i=1}^N \|Qy_{\sigma_i}(s) - Qg(s)\|_ Y^2\right) \mathrm ds\\
&\quad +  \frac{1}{2N} \sum_{i=1}^N \|P y_{\sigma_i}(T) - P g(T)\|_ Z^2,
\end{split}
\end{equation}
with~$\bfy_\varSigma\coloneqq(y_{\sigma_1},y_{\sigma_2},\dots,y_{\sigma_N})$ and each pair~$(y_{\sigma_i},u)$ subject to \eqref{eq:psys} with~$\sigma=\sigma_i$.

We can find the minimizing control input in feedback form~$u_\sigma(t) = K(t,y_\sigma(t))$,~$t \in [0,T]$, where~$K=K_\varSigma$ is affine on the  difference~$y_\sigma(t) - g(t)$, with a translation part depending on the residual of the target~$g$ when plugged into the uncontrolled system. We aim at a robust feedback~$K$, in the sense that by applying~$K$ for parameters~$\sigma \in \varSigma$, we should observe the desired tracking property towards~$g$. With such a feedback input, for any given fixed~$\sigma \in \fkS$, system \eqref{eq:psys} reads
\begin{align}
\dot{y}_\sigma(t) &= \clA_\sigma y_\sigma(t) + B K(t,y_\sigma(t)), \qquad y_\sigma(0) = y_{\circ},\notag
\end{align}

 We shall assume that the state space~$ H$ is a pivot space, that is, we will identify it with its continuous dual,~$ H= H'$. Further, we assume that~$ U$ is isomorphic to a closed subspace of~$ H$, so that we can consider~$ U$ as a pivot space as well, $ U= U'$. These identifications are common and convenient in (optimal) control applications (cf.~introductory notation in~\cite[Ch.~4]{TucsnakWeiss09}).

\subsection{Extended system.}\label{sS:extsystem}
For each~$\sigma \in \fkS$, it is assumed that~$\clA_\sigma$ is the infinitesimal  generator of  a $\clC_0$-semigroup $S_\sigma(t)_{t\ge 0}$ of bounded linear operators on $ H$.
The adjoint operator to ~$\clA_{\sigma}$ in $ H$ is denoted by~$\clA_{\sigma}^\ast$.
Equipping the domain~$\clD(\clA_{\sigma})$ of~$\clA_\sigma$ in~$H$ with the inner product~$\langle u,v \rangle_{\clD(\clA_\sigma)} \coloneqq \langle u, v\rangle_{ H} + \langle \clA_{\sigma} u , \clA_\sigma v\rangle_{ H}$,~$u,v\in \clD(\clA_\sigma)$, with the topology induced by the graph norm, $\clD(\clA_\sigma)$ becomes a Hilbert space, and $\clA_\sigma \in \clL(\clD(\clA_\sigma),  H)$.

Next, let us consider the Cartesian product~$ H^N \coloneqq \bigtimes_{i=1}^N  H$ with the usual inner product~$\langle \bfh,\widetilde \bfh\rangle_{ H^N} \coloneqq \sum_{i=1}^N ( h_i, \tilde h_i)_ H$, for $\bfh = (h_1,h_2,\ldots,h_N)$ and ~$\widetilde \bfh = (\tilde h_1,\tilde h_2, \ldots \tilde h_N)$. Further, we define the extension operator~$\clE\coloneqq \clE_N$ as
\begin{align}\notag
\clE: H &\to H^N,\qquad  z \mapsto  (z,z,\ldots,z).
\end{align}
 Its adjoint~$\clE^\ast:  H^N \to  H$ is given by
\begin{equation}\notag
\clE^\ast: H^N \to H,\qquad (w_1,w_2,\ldots,w_N) \;\mapsto \sum\limits_{i=1}^N w_{i}.
\end{equation}

Using the ensemble of operators~$\clA_{\sigma_i}$, $\sigma_i \in \varSigma$, we introduce the ensemble operator
\begin{align}
\bfA_{\varSigma}: \clD(\bfA_{\varSigma}) \subseteq  H^N &\to  H^N,\qquad
 w  \mapsto (\clA_{\sigma_1}w_{1},\clA_{\sigma_2}w_{2}, \ldots, \clA_{\sigma_N}w_{N}),\label{eq:extA}
\end{align}
where  $\clD(\bfA_{\varSigma}) = \bigtimes_{i=1}^N \clD(\clA_{\sigma_i})$.
We also define, for a given Hilbert space~$X$ and an operator~$L\in\clL(H,X)$,
\begin{equation}
L_\rme\in\clL(H^N,X^N),\qquad (w_1,w_2,\ldots,w_N)\coloneqq (Lw_1,Lw_2,\ldots,Lw_N).\notag
\end{equation}
Now, we can reformulate the problem of minimizing~\eqref{eq:obj} with each~$(y_{\sigma_i},u)$ subject to~\eqref{eq:psys}, for all $\sigma\in\varSigma$, as:
\begin{subequations}\label{eq:OPext}
\begin{align}
&\mbox{minimize }&&\clJ(\bfy_{\varSigma},u) = \int_0^T \!\left(\frac12\|u(s)\|_ U^2+ \frac{1}{2N} \left\|  Q_\rme\bfy_{\varSigma}(s) - Q_\rme\clE  g(s)\right\|_{Y^N}^2 \right)\rmd s \notag\\
&&&\hspace{5em}+ \frac{1}{2N} \left\|( P_\rme\bfy_\varSigma(T) - P_\rme\clE  g(T)\right\|^2_{Z^N}
\label{eq:extobj}\\
&\mbox{subject to }&&\dot{\bfy}_{\varSigma}(t) = \bfA_{\varSigma} \bfy_{\varSigma}(t) + \bfB u(t),\qquad
\bfy_{\varSigma}(0) = \bfy_{\circ},\label{eq:extsys}
\end{align}
\end{subequations}
where $\bfy_{\circ} = \clE y_\circ \in  H^N$ and $\bfB = \clE B:  U \to  H^N$.

We observe that~$\bfA_\varSigma$, defined in~\eqref{eq:extA}, is the infinitesimal generator of the
$\clC_0$-semigroup
\begin{align}
\bfS_{\varSigma}(t):  H^N \to  H^N,\qquad z \mapsto (S_{\sigma_1}(t)z_{1},S_{\sigma_2}(t)z_{2},\ldots,S_{\sigma_N}(t)z_{N}) \notag
\end{align}
of bounded linear operators on~$H^N$.


\subsection{Optimal control input for the extended system.}\label{sS:optim-ext}
Based on existing results for Riccati equations, in this section we ensure the existence and uniqueness of a feedback control  for problem~\eqref{eq:OPext}.

For this purpose we introduce the cone~$\Omega( H^N)$  of bounded, linear, self-adjoint, and nonnegative operators in $ H^N$ endowed with the norm of $\clL( H)$.
The Riccati operators will be sought as strongly continuous operator-valued functions in the set~$\clS \coloneqq \clC_s([0,T],\Omega( H^N))$, which is endowed with the topology of strong convergence, i.e.,~$F_n \to F$ if and only if~$\forall x \in  H^N$ there holds~$F_n x \to Fx$ in~$\clC([0,T]; H^N)$, see e.g.,~\cite[ Part IV, Section 2.1]{bensoussan2007representation}

For simplicity, we shall transform our problem of tracking~$g$ to a problem of tracking~$0$ (subject to an inhomogeneous state equation). In this manner we can more directly profit from existing theory on Riccati equations.

Let the target satisfy~$g \in W^{1,2}(0,T; H) \bigcap L^2(0,T;\bigcap_{i=1}^N\clD(\clA_{\sigma_i}))$. Denoting
\begin{equation}\label{eq:y-to-x}
\bfx_\varSigma \coloneqq \bfy_\varSigma - \clE g\quad\mbox{and}\quad\bff \coloneqq \bfA_\varSigma \clE g- \clE \dot{g} ,
\end{equation}
 problem~\eqref{eq:OPext} becomes the problem
\begin{subequations}\label{eq:OPextx}
\begin{align}
&\mbox{minimize }&&
\clJ(\bfx_{\varSigma},u) = \int_0^T \left( \frac{1}{2}\|u(s)\|_{ U}^2+ \frac{1}{2N} \left\|  Q_\rme\bfx_{\varSigma}(s) \right\|_{Y^N}^2 \right) \,\rmd s\notag\\
&&&\hspace{5em}+ \frac{1}{2N} \left\| P_\rme\bfx_{\varSigma}(T)\right\|^2_{ Z^N}\label{eq:extobjx}
\\
&\mbox{subject to }&&
\dot{\bfx}_{\varSigma}(t) = \bfA_{\varSigma} \bfx_{\varSigma}(t) + \bff(t) + \bfB u(t),\qquad
\bfx_{\varSigma}(0) = \bfx_{\circ},\label{eq:extsysx}
\end{align}
\end{subequations}
with~$\bfx_{\circ} \coloneqq \bfy_{\circ} - \clE g(0)$.

Consider, for time~$t\in(0,T)$, the operator differential Riccati equation
\begin{align}
\dot{\bfPi}_\varSigma &= \bfPi_\varSigma \bfA_{\varSigma} + \bfA_{\varSigma}^\ast \bfPi_\varSigma - \bfPi_\varSigma\bfB \bfB^\ast \bfPi_\varSigma+ \frac1N Q_\rme^*Q_\rme,\qquad
\bfPi_\varSigma(0) = \frac{1}{N} P_\rme^*P_\rme.\label{eq:extRiccati}
\end{align}
The dynamics equation in~\eqref{eq:extRiccati} is understood in the sense that for any $\bfx,\bfy \in \clD(\bfA_\varSigma)$ the function $t\mapsto\langle \bfPi_\varSigma(t)\bfx, \bfy\rangle_{ H^N}$ is differentiable  and satisfies for almost all~$t\in(0,T)$,
\begin{align}
\frac{\mathrm d}{\mathrm dt} \langle \bfPi_\varSigma(t) \bfx,\bfy \rangle_{ H^N} &= \langle \bfPi_\varSigma(t) \bfA_\varSigma \bfx, \bfy \rangle_{ H^N} + \langle\bfPi_\varSigma(t)\bfx, \bfA_\varSigma \bfy \rangle_{ H^N}\notag\\
&\quad- \langle \bfB^\ast\bfPi_\varSigma(t) \bfx, \bfB^\ast \bfPi_\varSigma(t) \bfy \rangle_{U} + \frac{1}{N} \langle  Q_\rme\bfx,  Q_\rme\bfy \rangle_{Y^N}.\notag
\end{align}
 From~\cite[Thm.~2.1, Part IV, Ch.~1]{bensoussan2007representation} we know that~\eqref{eq:extRiccati} admits a unique solution in~$ \mathcal{S}$.
In the following theorem we recall from~\cite[Thm.~7.1, Part IV, Ch.~1]{bensoussan2007representation}, how to construct the optimal pair of feedback control and corresponding state of~\eqref{eq:extobjx} subject to~\eqref{eq:extsysx} (and thus for~\eqref{eq:extobj} subject to~\eqref{eq:extsys}).
\begin{theorem}
 Let~$\bfPi_\varSigma$ denote the unique solution of~\eqref{eq:extRiccati} in~$\mathcal{S}$.
Then, there exists a unique  minimizer~$(\bfx,u)$  for~\eqref{eq:OPextx}. This optimal pair satisfies, for~$t\in(0,T)$,
\begin{itemize}
\item[1.] $ u(t)$ is given in feedback form by
\begin{align}\label{eq:uopt}
 u(t) = - \bfB^\ast \left( \bfPi_\varSigma(T-t) \bfx(t) + \bfh(t) \right);
\end{align}
\item[2.] $ \bfx$ is the mild solution to the closed-loop system
\begin{align}
\dot{ \bfx}(t) &= \left(\bfA_\varSigma - \bfB\bfB^\ast \bfPi_\varSigma(T-t) \right) \bfx(t) - \bfB\bfB^\ast \bfh(t) + \bff(t),\qquad \bfx(0) = \bfx_\circ;\label{eq:xopt}
\end{align}
where
\begin{align}
-\dot{\bfh}(t) &= \left(\bfA_\varSigma^\ast - \bfPi_\varSigma(T-t)\bfB\bfB^\ast\right) \bfh(t) + \bfPi_\varSigma(T-t)\bff(t),\qquad
\bfh(T) = \mathbf{0};\label{eq:h}
\end{align}
\item[3.] the optimal cost is given by
\begin{equation}
\label{eq:mincost}
\begin{split}
\clJ(\bfx, u) &= \frac{1}{2} \langle \bfPi_\varSigma(T) \bfx_\circ,\bfx_\circ\rangle_{ H^N} + \langle \bfh(0), \bfx_\circ\rangle_{ H^N} \\
&\quad+ \int_0^T \left( \langle \bfh(s), \bff(s) \rangle_{ H^N} - \frac{1}{2}\|\bfB^\ast \bfh(s)\|_{ U}^2\right) \mathrm ds.
\end{split}
\end{equation}
\end{itemize}
\end{theorem}

\subsection{From the extended system to the original one}\label{sS:ext-2-ori}
By construction of the feedback~$\bfPi_\varSigma$, we expect that~$\left\|  Q_\rme\bfx_{\varSigma} \right\|_{L^2(0,T;Y^N)}^2 = \left\|  Q_\rme\bfy_{\varSigma} - Q_\rme\clE  g\right\|_{L^2(0,T;Y^N)}^2 $ will be small. Consequently, we can expect that the component~$Q(y_{\sigma} -  g)$ of the difference~$y_{\sigma} -  g$ to the target~$g$ will be small for all~$\sigma\in\varSigma$.
By solving the extended system we obtain a tracking control input for all~$\sigma\in\varSigma$.
In our context this input is of auxiliary nature, indeed  this auxiliary extended state is not available in practice. Rather the goal of this section is to propose a feedback depending only on the state of the original unknown system.

We  define the feedback input operator~$K_\varSigma\colon [0,T]\times H\to U$ which is constructed by means of ~$\bfK_\varSigma(t,\bfz)\coloneqq- \bfB^\ast \left( \bfPi_\varSigma(T-t) \bfz + \bfh(t) \right)\colon [0,T]\times H^N\to U$ computed for the extended system,  by 
\begin{align}\label{eq:robFB}
K_{\varSigma}(z)\coloneqq K_{\varSigma}(t,z) &\coloneqq  -  \bfB^\ast \left(\bfPi_\varSigma(T-t) \clE z + \bfh(t) \right),\quad\mbox{for}\quad t\in[0,T].
\end{align}

\begin{remark}
In~\eqref{eq:robFB}, the ``definition~$K_{\varSigma}(z)\coloneqq K_{\varSigma}(t,z)$'' simply means that sometimes, for simplicity of the exposition, we will omit the dependence of~$K_{\varSigma}$ on~$t$.
\end{remark}
Therefore, we arrive at the closed-loop system
\begin{align}\label{eq:xsig-robFB}
\dot{x}_{\varSigma,\sigma}&= \clA_{\sigma} x_{\varSigma,\sigma} + B K_{\varSigma}(x_{\varSigma,\sigma}) + \clA_\sigma g- \dot{g},\qquad
x_{\varSigma,\sigma}(0) = x_\circ.
\end{align}
 Defining~$y_{\varSigma,\sigma}\coloneqq x_{\varSigma,\sigma} + g$, we find~$\dot{y}_{\varSigma,\sigma}= \clA_{\sigma} y_{\varSigma,\sigma} + B  K_{\varSigma}(y_{\varSigma,\sigma}- g)$, hence
\begin{subequations}\label{eq:ysig-robFB}
\begin{align}
\dot{y}_{\varSigma,\sigma}&=\clA_{\sigma} y_{\varSigma,\sigma} + B K_{\varSigma}^{[1]}y_{\varSigma,\sigma}+ B K_{\varSigma}^{[0]},\qquad
y_{\varSigma,\sigma}(0) = y_\circ,
\intertext{with}
 K_{\varSigma}^{[1]}&=-  \bfB^\ast \bfPi_\varSigma(T-t) \clE\quad\mbox{and}\quad
 K_{\varSigma}^{[0]}\coloneqq   \bfB^\ast \left(\bfPi_\varSigma(T-t) \clE g(t)  - \bfh(t) \right),
\end{align}
\end{subequations}

Since  the linear part $B K_{\varSigma}^{[1]}$ of the affine feedback is strongly continuous, that is,
$-B \bfB^\ast \bfPi_\varSigma(T-\cdot) \clE z \in \clC([0,T]; H)$ for each~$ z\in  H$, and  the translation  term  $B K_{\varSigma}^{[0]}$ is in $L^2(0,T; H)$, the above closed-loop  system~\eqref{eq:ysig-robFB} has a unique solution $y_{\varSigma,\sigma} \in \clC([0,T]; H)$ for each $\sigma \in \varSigma$, (see, e.g., \cite[Prop.~3.4, Part II, Ch.~1]{bensoussan2007representation}).  Consequently, there is a unique solution~$x_{\varSigma,\sigma} \in \clC([0,T]; H)$ for system~\eqref{eq:xsig-robFB}, for any given~$\sigma \in \varSigma$.

Finally, note that the feedback $ K_{\varSigma}$ can also be applied if the true parameter is not a member of the training set $\varSigma$, provided that $g\in W^{1,2}(0,T; H)\cap L^2(0,T;\mathcal{D}(\mathcal{A}_\sigma))$  (cf., \eqref{eq:y-to-x}).  This will be the generic case in the following sections.

\subsection{Order of sequence of training parameters}\label{sS:ensemble-set}
By construction the matrix~$\bfA_\Sigma$, defining the free dynamics of the extended auxiliary system as in~\eqref{eq:extsys}, depends on the order of the training parameters in the sequence~$\varSigma=(\sigma_i)_{ i=1}^{N}$. In spite of this fact, we show that the resulting feedback input~$K_{\varSigma}(z)$ as in~\eqref{eq:robFB}, for the original system, is independent of that order. In this sense, we can speak about {\em  set} of training parameters, instead of {\em sequence} of training parameters.
Indeed, let~$\vartheta\in\bbR^{N\times N}$ be a permutation matrix~$\vartheta\colon\bbR^N\to\bbR^N$, and let~$\Theta=\Theta(\vartheta)\in\clL(H)^{N\times N}$ be the permutation~$\Theta\colon H^N\to H^N$ constructed as follows: the entries~$1$ of~$\vartheta$ are replaced by the identity operator~$\Id=\Id_H$ in~$H$ and the entries~$0$ are replaced by the zero operator~$\zero=\zero_H$ in~$H$. 

As an example, in case~$N=3$,
\begin{align}
\mbox{if}\quad\vartheta=\begin{bmatrix}
0&0&1\\ 1&0&0\\0&1&0
\end{bmatrix},\quad\mbox{then}\quad\Theta(\vartheta)=\begin{bmatrix}
\zero_H&\zero_H&\Id_H\\ \Id_H&\zero_H&\zero_H\\\zero_H&\Id_H&\zero_H
\end{bmatrix}.\notag
\end{align}

Identifying the sequence~$\varSigma$ with a column vector in~$\bbR^{N\times 1}$, we permute the parameters as~$\varSigma\to \vartheta\varSigma$. For the permuted/reordered vector, the extended matrix will read
\begin{align}
\bfA_{\vartheta\varSigma}=\Theta\bfA_{\varSigma}\Theta^\top\notag
\end{align}
where~$\Theta^\top\coloneqq\Theta(\vartheta^\top)$. Recall that, since~$\vartheta$ and~$\Theta$ are permutations we have~$\vartheta^\top=\vartheta^{-1}$ and~$\Theta^\top=\Theta^{-1}$. By~\eqref{eq:extRiccati} we find that
$\bfR_\vartheta\coloneqq\Theta{\bfPi}_\varSigma\Theta^\top$ solves
\begin{align}
\dot\bfR_\vartheta&=\Theta\dot{\bfPi}_\varSigma\Theta^\top = \Theta\bfPi_\varSigma \bfA_{\varSigma}\Theta^\top + \Theta\bfA_{\varSigma}^\ast \bfPi_\varSigma\Theta^\top - \Theta\bfPi_\varSigma\bfB \bfB^\ast \bfPi_\varSigma\Theta^\top+ \frac1N \Theta Q_\rme^*Q_\rme\Theta^\top\notag\\
&= \bfR_\vartheta\bfA_{\vartheta\varSigma} +\bfA_{\vartheta\varSigma}^\ast \bfR_\vartheta -  \bfR_\vartheta\Theta\bfB \bfB^\ast \Theta^\top\bfR_\vartheta + \frac1N \Theta Q_\rme^* Q_\rme\Theta^\top,\notag\\
\bfR_\vartheta(0) &= \frac{1}{N} \Theta P_\rme^* P_\rme\Theta^\top\notag.
\end{align}
Now, observe that~$\Theta\clE=\clE$, which gives us the analogue of~\eqref{eq:extRiccati}
\begin{align}
\dot\bfR_\vartheta &=  \bfR_\vartheta\bfA_{\vartheta\varSigma} +\bfA_{\vartheta\varSigma}^\ast \bfR_\vartheta -  \bfR_\vartheta\bfB \bfB^\ast \bfR_\vartheta+ \frac1N Q_\rme^* Q_\rme,\qquad
\bfR_\vartheta(0) =  \frac{1}{N}  P_\rme^* P_\rme.\notag
\end{align}
Therefore~$\bfR_\vartheta$ is the solution of the Riccati equation for the permuted sequence of parameters.
Consequently, the feedback input in~\eqref{eq:robFB} will read, since we also have~$\Theta^\top\clE=\clE$
\begin{align}
K_{\vartheta\varSigma}(z)&=  -  \bfB^\ast \left(\bfR_\vartheta(T-t) \clE z + \bfh_\vartheta(t) \right)=-  \bfB^\ast \left({\bfPi}_\varSigma(T-t) \clE z + \bfh_\vartheta(t) \right),\notag
\end{align}
where~$\bfh_\vartheta$ satisfies the analogue of~\eqref{eq:h},
\begin{align}
-\dot{\bfh_\vartheta}(t) &= \left(\bfA_{\vartheta\varSigma}^\ast - \bfR_\vartheta(T-t)\bfB\bfB^\ast\right) \bfh_\vartheta(t) + \bfR_\vartheta(T-t)\bff_\vartheta(t),\qquad\bfh_\vartheta(T)=0 ,\label{eq:h-perm}
\end{align}
with the analogue of~$\bff$ in~\eqref{eq:y-to-x},
\begin{equation}
\bff_\vartheta \coloneqq \bfA_{\vartheta\varSigma} \clE g- \clE \dot{g}.\notag
\end{equation}

Thus, to show that~$K_{\vartheta\varSigma}(z)=K_{\varSigma}(z)$, it is enough to show that
$\bfB^*\bfh_\vartheta=\bfB^* \bfh$.

By~\eqref{eq:h-perm}, using~$\Theta\bfB=\Theta\clE B=\clE B=\bfB$ and~$\Theta^\top\bfB=\Theta^\top\clE B=\clE B=\bfB$,
\begin{align}
-\Theta^\top\dot{\bfh_\vartheta}(t) &= \left(\Theta^\top\bfA_{\vartheta\varSigma}^\ast - \Theta^\top\bfR_\vartheta(T-t)\bfB\bfB^\ast\right) \bfh_\vartheta(t) +\Theta^\top \bfR_\vartheta(T-t)\bff_\vartheta(t),\notag\\
&= \left(\bfA_{\varSigma}^\ast\Theta^\top - {\bfPi}_\varSigma(T-t)\Theta^\top\bfB\bfB^\ast\right) \bfh_\vartheta(t) +{\bfPi}_\varSigma(T-t)\Theta^\top\bff_\vartheta(t),\notag\\
&= \left(\bfA_{\varSigma}^\ast - {\bfPi}_\varSigma(T-t)\bfB\bfB^\ast\right) \Theta^\top\bfh_\vartheta(t) +{\bfPi}_\varSigma(T-t)\bff(t).\label{dyn-htheta}
\end{align}

Now, by~\eqref{eq:h} and~\eqref{dyn-htheta}, we find that~$d\coloneqq \Theta^\top\bfh_\vartheta-\bfh$
solves the linear system
\begin{align}
\dot d(t) &=  -\left(\bfA_{\varSigma}^\ast - {\bfPi}_\varSigma(T-t)\bfB\bfB^\ast\right) d ,\qquad d(0)=0,\notag
\end{align}
hence~$\Theta^\top \bfh_\vartheta(t)- \bfh(t)=d(t)=0$ for~$t\in[0,T]$, which implies
$\bfB^*\bfh_\vartheta=\bfB^*\Theta^\top\bfh_\vartheta=\bfB^* \bfh$.

\section{Optimal controls and costs}\label{S:optimCost}

In this section we investigate the cost associated with the feedback~\eqref{eq:robFB} and compare it to the optimal cost associated with the input~\eqref{eq:uopt} for the extended system~$(\bfA_\varSigma,\bfB)$. Further, we consider a comparison with the optimal cost for system~$(A_\sigma,B)$, corresponding to the case that the true parameter~$\sigma$ is known.
For such comparisons, we shall make additional assumptions on the family of operators~$\{\clA_\sigma\mid \sigma\in\fkS\}$.

We will assume that we have another separable Hilbert space~$V$ that is continuously and densely embedded in~$H$, which leads to the Gelfand triplet~$V \subset H \subset V'$. We also assume to be given a family of continuous bilinear forms~$a(\sigma;\Bigcdot,\Bigcdot)$, parametrized by~$\sigma\in\fkS$, each form being~$V$--$H$ coercive, more precisely,
\begin{equation}\label{eq:aux1}
\begin{split}
&\mbox{there exists}\quad (\rho, \theta)\in\bbR\times\bbR^+\quad\mbox{such that, for all}\quad( \sigma,v) \in \fkS\times V,\\
&\mbox{there holds}\quad a(\sigma;v,v) + \rho \|v\|_{H}^2 \ge \theta \|v\|_V^2.
\end{split}
\end{equation}
We associate with~$a(\sigma;\Bigcdot,\Bigcdot)$ the operator~$\clA_\sigma$ defined as
\begin{align}
\label{eq:bilinearform}
\begin{split}
\clD(\clA_\sigma) &\coloneqq \{v\in V\,\mid\,w\mapsto a(\sigma;v,w) \mbox{ is }H\mbox{-continuous}\},\\
\langle \clA_\sigma v,w\rangle_{ H} &\coloneqq -a(\sigma;v,w),\quad\forall v\in \clD(\clA_\sigma),\quad\forall w\in V.
\end{split}
\end{align}
The operators~$\clA_\sigma$ are closed and densely defined in~$H$ and can be uniquely extended to operators~$\clA_\sigma \in \clL(V,V')$. For each $\sigma \in \fkS$,~$\clA_\sigma$ generates an analytic semigroup~$S_\sigma(t)$ on~$H$, which is exponentially bounded (i.e., $\|S_\sigma(t)\|_{\clL(H,H)} \le e^{\rho t}$).

The following assumption will be made throughout the remainder of the paper.
\begin{assumption}\label{A:domindsig}
There exists a family~$a$ of bilinear forms satisfying~\eqref{eq:aux1} such that the operators~$\clA_\sigma$ are characterized by~\eqref{eq:bilinearform}. Furthermore, $\clD(\clA_\sigma)$ is independent of~$\sigma \in \fkS$ and~$\clD(\clA_\sigma)=\clD(\clA_\sigma^\ast)$ for all~$\sigma \in \fkS$.
\end{assumption}

By Assumption~\ref{A:domindsig},  we can introduce the common domain
\begin{equation}\label{eq:DAsame}
\clD_\clA=\clD(\clA_{\sigma})=\clD(\clA_{\sigma}^*),\quad\mbox{for all}\quad\sigma \in \fkS.
\end{equation}
In particular~$\clD(\bfA_\varSigma) = \clD_\clA^N$ is independent of the ensemble~$\varSigma\subset\fkS$.

Now, recalling the short notation in~\eqref{eq:def-Space_T}, it is known (see, e.g.,~\cite[Thm.~1.1, Part II, Ch.~2]{bensoussan2007representation}) that for $f \in V_T'$ and $x_\circ \in H$ there is a unique solution $x\in W_T(V,V')\hookrightarrow \clC([0,T];H)$ of $\dot{y} = \clA y + f$. From~$\clE g\in W_T(\clD_\clA^N,H^N)$ (as assumed in Sect.~\ref{sS:optim-ext}) we find
\begin{align}
f &= \clA_\sigma g - \dot{g} \in L^2(0,T;H)\quad\mbox{and}\quad
\bff = \bfA_\varSigma \clE g - \clE\dot{g} \in L^2(0,T;H^N).\label{eq:bff}
\end{align}
Recall  that (from, e.g.,~\cite[Thm.~1.1, Part II, Ch.~2]{bensoussan2007representation}), we have that
\begin{subequations}\label{eq:iso_dyn+ic}
\begin{equation}
\bfy \mapsto ( \dot{\bfy} - \bfA_\varSigma \bfy,\bfy(0)),\quad W_T(V^N,(V')^N) \to (V_T^N)' \times H^N
\end{equation}
is an isomorphism. Hence, there exists a constant $C_W>0$ such that
\begin{equation}
\|\bfy\|_{W_T(V^N,(V^N)')} \leq C_W \|(\dot{\bfy} - \bfA_\varSigma \bfy,\bfy(0))\|_{(V_T^N)'\times H^N}.
\end{equation}
\end{subequations}

\subsection{Comparing optimal costs}
Let us fix~$\sigma\in\fkS$ and consider the problem:
\begin{subequations}\label{eq:OP1sig}
\begin{align}
&\mbox{minimize }&&\!\clJ(\clE x_\sigma,u_\sigma) = \frac12 \int_0^T\! \left( \|Q x_\sigma(t)\|^2_{H} +  \|u_\sigma(t)\|^2_{ U} \right) \mathrm dt + \frac{1}{2} \|P x_\sigma(T)\|_{H}^2,\label{eq:smallsysobj}\\
&\mbox{subject to }&&
\dot{x}_\sigma(t) = \clA_\sigma x_\sigma(t) + B u_\sigma(t) + f(t),\qquad
x_\sigma(0) = x_{\circ}.\label{eq:smallsys}
\end{align}
\end{subequations}

In the following we will use the notation~$\bfA_\sigma$ to denote the operator defined as in~\eqref{eq:extA} with the same parameter~$\sigma_i=\sigma$ in each component.

\begin{lemma}\label{lem:objdiff1}
Given a parameter~$\sigma \in \fkS$, let~$(x_\sigma,u_\sigma)$ be the unique minimizer of problem~\eqref{eq:OP1sig} and let~$(\bfx_\varSigma, u_\varSigma)$ be the unique minimizer of problem~\eqref{eq:OPextx}.
Then, there holds
\begin{equation}\notag
\clJ(\bfx_\varSigma - \clE x_\sigma, u_\varSigma - u_\sigma) \le \mathfrak{C}_1 \|\bfA_\varSigma - \bfA_\sigma\|^2_{\clL(V^N,(V^N)')},
\end{equation}
where
\begin{equation}\label{eq:fkC}
\mathfrak{C}_1\coloneqq \frac{1}{2} C_W \left(1+\mathfrak{C}_{\bfp}+C_WC_H^2\|B\|_{\clL( U, H)}^2+2\mathfrak{C}_{\bfp}^2 C_W\right)
\|(x_{\sigma}+g, p_{\sigma})\|_{V_T\times V_T}^2
\end{equation}
with~$C_W$ as in~\eqref{eq:iso_dyn+ic}, $C_H\coloneqq \|\Id\|_{\clL(H^N,(V^N)')}$ and
\begin{equation}\notag
\mathfrak{C}_{\bfp} = \max\{1,C_V\|Q_\rme\|_{\clL(H^N,Y^N)},C_1\|P_\rme\|_{\clL(H^N,Z^N)}\},
\end{equation}
where~$C_1\coloneqq \|\Id\|_{\clL(W_T(V^N,(V^N)'),\clC([0,T],H^N)}$~and~$C_V\coloneqq \|\Id\|_{\clL(V^N,H^N)}$.
\end{lemma}

\begin{proof}
 Given~$x_\circ \in H$,
the  optimality conditions for~\eqref{eq:OPextx}, with~$\bfx_\circ= \clE x_\circ$, are
\begin{subequations}\label{eq:optcondext}
\begin{align}
\dot{\bfx}_{\varSigma} &= \bfA_{\varSigma} \bfx_{\varSigma} + \bff+ \bfB u_\varSigma,\qquad && \bfx_{\varSigma}(0) = \clE x_\circ,\label{eq:optcondAext}\\
\dot{\bfp}_{\varSigma}&= -\bfA_{\varSigma}^\ast \bfp_{\varSigma} - \frac1N  Q_\rme^* Q_\rme\bfx_{\varSigma}\,\qquad&& \bfp_{\varSigma}(T) = \frac{1}{N} P_\rme^*P_\rme \bfx_{\varSigma}(T), \label{eq:optcondpext}\\
u_{\varSigma} &= -\bfB^\ast \bfp_{\varSigma},\label{eq:optcondEext}
\end{align}
\end{subequations}
and  for~\eqref{eq:OP1sig}, they are
\begin{align}
\dot{x}_{\sigma} &= \clA_{\sigma} x_{\sigma} + f + B u_\sigma,\qquad && x_{\sigma}(0) = x_{\circ},\notag\\
\dot{p}_{\sigma}&= -\clA_{\sigma}^\ast p_{\sigma} -  Q^*Q x_{\sigma},\qquad&& p_{\sigma}(T) =   P^*P x_\sigma(T), \notag\\
u_\sigma &= -B^\ast p_{\sigma}.\notag
\end{align}
The reader is reminded that the factor $\frac1N$, in \eqref{eq:optcondpext}, accounts for taking the sum over the ensemble~$\varSigma = \{\sigma_i \,|\, 1\le i \le N\}$; see~\eqref{eq:extobjx}.

Defining
\begin{equation}\notag
\delta\bfx \coloneqq \bfx_{\varSigma} - \clE x_{\sigma},\!\qquad \delta\bfp \coloneqq \bfp_{\varSigma} - \frac1 N \clE p_{\sigma},\!\qquad
 \delta u  \coloneqq u_{\varSigma} - u_{\sigma},\!\qquad\delta\bfA\coloneqq\bfA_{\varSigma} - \bfA_{\sigma},
\end{equation}
 we obtain
\begin{subequations}\label{eq:optcondiff}
\begin{align}
\dot{\delta\bfx} &= \bfA_{\varSigma} \delta\bfx + \delta\bfA \clE \left(x_{\sigma}+g\right) + \bfB \delta u, &&\qquad\delta\bfx(0) = 0,\label{eq:optcondiffA}\\
\dot{\delta\bfp} &= -\bfA^\ast_{\varSigma} \delta\bfp - \delta\bfA^\ast \frac1 N \clE p_{\sigma} - \frac1 N Q_\rme^*Q_\rme \delta \bfx, \label{eq:optcondiffC}, &&\qquad\delta\bfp(T) = \frac{1}{N} P_\rme^*P_\rme\delta \bfx(T),\\
\delta u &= -\bfB^\ast\delta\bfp.&&\label{eq:optcondiffE}
\end{align}
\end{subequations}
Moreover, we have~$\bfx_\varSigma \in W_T(V^N,(V^N)')$ and~$x_\sigma \in W_T(V,V')$, and due to~$\delta \bfx(0) \in V^N$, we obtain~$\delta \bfx \in W_T(\clD_\clA^N,H^N)\subset \clC([0,T];V^N)$, and thus~$\delta \bfp \in W_T(\clD_\clA^N,H^N)\subset \clC([0,T];V^N)$ (cf.~\cite[Thm.~1.4, Part II, Ch.~2]{bensoussan2007representation}).

Hence,  we find the identity
\begin{align}
-\langle \dot{\delta\bfp}, \delta\bfx \rangle_{H_T^N} &= \langle  \bfA_{\varSigma}^\ast \delta\bfp, \delta\bfx \rangle_{H_T^N}+\! \langle \delta\bfA^\ast \! \frac1N \clE p_{\sigma}, \delta\bfx \rangle_{(V_T^N)',V_T^N} + \!\frac1N \| Q_\rme \delta\bfx\|^2_{Y_T^N} \!\!\!\label{eq:diffA}
\intertext{and, using~$\langle\delta\bfp(0), \delta\bfx(0) \rangle_{H^N}=0$, we also have}
-\langle\delta\bfx ,\dot{\delta\bfp}\rangle_{H_T^N} &=\langle\dot{\delta\bfx},\delta\bfp\rangle_{H_T^N}-\frac{1}{N} \| P_\rme\delta\bfx(T)\|^2_{Z^N}\label{eq:diffB}\\
&= \langle \bfA_{\varSigma}\delta\bfx,\delta\bfp \rangle_{H_T^N}+ \langle \delta\bfA\clE (x_{\sigma}+g), \delta\bfp \rangle_{(V_T^N)',V_T^N}+ \langle \bfB\delta u,\delta\bfp \rangle_{H_T^N} \notag\\&\quad  - \frac{1}{N} \|P_\rme\delta\bfx(T)\|^2_{Z^N}\notag
\end{align}
Subtracting~\eqref{eq:diffA} from~\eqref{eq:diffB}, and using~\eqref{eq:optcondiffE}, lead us to
\begin{align}
\delta\clJ\coloneqq~&\frac1 N \|Q_\rme\delta\bfx\|^2_{Y_T^N} +   \| \delta u\|^2_{ U_T} + \frac{1}{N} \|P_\rme\delta\bfx(T)\|^2_{Z^N}\label{eq:deltaJ}\\
=~& -\langle \delta\bfA^\ast \frac1 N \clE p_{\sigma}, \delta\bfx\rangle_{(V_T^N)',V_T^N}+ \langle \delta\bfA\clE (x_{\sigma}+g), \delta\bfp \rangle_{(V_T^N)',V_T^N}\notag\\
\le~& \|\delta\bfA\|_{\clL(V^N,(V^N)')} \Bigl( \|\clE( x_{\sigma}+g)\|_{V_T^N}\,\|\delta\bfp\|_{V_T^N} + \frac1N \|\clE p_{\sigma}\|_{V_T^N}\,\|\delta\bfx\|_{V_T^N}\Bigr).\label{eq:deltabound}
\end{align}

Next, we use a duality argument to estimate the norm of $\delta \bfp$. Let us denote by~$\clB^{\clX}\coloneqq\{h\in \clX\mid \|h\|_{\clX}\le1\}$ the unit ball in a given Hilbert space~$\clX$. Let~$\bfb\in \clB^{(V_T)'}$ be arbitrary and let~$\bfy=\bfy(\bfb)$ be the solution of
\begin{align}
\dot{\bfy} &= \bfA_\varSigma \bfy  + \bfb,\qquad\bfy(0) = \boldsymbol{0},\notag
\end{align}
for time~$t\in(0,T)$. Then, we have
\begin{align}
&\|\delta \bfp\|_{V_T^N} = \sup_{\bfb\in\clB^{(V_T)'}} \langle \delta \bfp, \bfb\rangle_{V_T^N,(V_T^N)'}
=\sup_{\bfb\in\clB^{(V_T)'}} \langle \delta \bfp, \dot{\bfy} - \bfA_{\varSigma} \bfy  \rangle_{V_T^N,(V_T^N)'}\notag\\
&= \sup_{\bfb\in\clB^{(V_T)'}}\big( \langle -\dot{\delta \bfp} - \bfA_{\varSigma}^\ast \delta \bfp, \bfy \rangle_{(V_T^N)',V_T^N}+\langle \delta\bfp(T), \bfy(T) \rangle_{H^N}\big)\notag\\
&= \sup_{\bfb\in\clB^{(V_T)'}}\Big(\langle \delta\bfA^\ast \frac1N \clE p_{\sigma},  \bfy \rangle_{(V_T^N)',V_T^N} + \frac1 N \langle Q_\rme^* Q_\rme\delta \bfx, \bfy\rangle_{H_T^N}\notag \\\notag&\quad\quad\quad\quad+\langle \frac{1}{N}P_\rme^* P_\rme\delta\bfx(T), \bfy(T) \rangle_{H^N}\Big),\notag
\end{align}
where we used~\eqref{eq:optcondiffC}. Since~$W_T(V^N,(V^N)')\subset \clC([0,T];H^N)$ (see, e.g.,~\cite[Thm.~1, Ch.~XVIII]{dautray2000lions}), we have~$\|\bfy(T)\|_{H^N} \le C_1\|\bfy\|_{W_T(V^N,(V^N)')}$, for some constant~$C_1>0$. By noticing that~$\|v\|_{V_T} \le \|v\|_{W_T(V,V')}$ and recalling  the isomorphism in~\eqref{eq:iso_dyn+ic}, we obtain
\begin{align}
\|\delta \bfp\|_{V_T^N} &\le \sup_{\|\bfb\|_{(V_T^N)'}\le1} \mathfrak{C}_{\bfp}\|\bfy\|_{W_T(V^N,(V^N)')}\bigg( \frac1 N \|\delta\bfA\|_{\clL(V^N,(V^N)')}\, \|\clE p_{\sigma}\|_{V_T^N} \notag\\
&\quad\quad\quad\quad+ \frac1 N \|  Q_\rme\delta \bfx\|_{ Y_T^N} +\frac{1}{N}\| P_\rme\delta\bfx(T)\|_{ Z^N}\bigg)\notag\\
&\hspace{-1em}\le \frac1 N C_W\mathfrak{C}_{\bfp} \bigg(  \|\delta\bfA\|_{\clL(V^N,(V^N)')}\, \|\clE p_{\sigma}\|_{V_T^N} +\|Q_\rme\delta \bfx\|_{Y_T^N} +\|P_\rme\delta\bfx(T)\|_{Z^N}\bigg),
\label{eq:deltap}
\end{align}
where~$C_W$ is as in~\eqref{eq:iso_dyn+ic} and $\mathfrak{C}_{\bfp} = \max{(1,C_V\|Q_\rme\|_{\clL(H^N,Y^N)},C_1\|P_\rme\|_{\clL(H^N,Z^N)}})$ with~$C_V\coloneqq \|\Id\|_{\clL(V^N,H^N)}$. 

Using~\eqref{eq:optcondiffA}, we will next estimate the term $\|\delta\bfx\|_{V_T^N}$ in~\eqref{eq:deltabound}:
\begin{align}
\|\delta\bfx\|_{V_T^N} &\le \|\delta\bfx\|_{W(0,T;V^N,(V^N)')} \le C_W \|(\bfA_{\varSigma} - \bfA_{\sigma}) \clE \left(x_{\sigma}+g\right) + \bfB \delta u\|_{(V_T^N)'}\label{eq:deltaxV}\\
&\le  C_W \|\delta \bfA\|_{\clL(V^N,(V^N)')} \|\clE(x_\sigma + g)\|_{V_T^N} + C_WC_{H}\|\bfB\|_{\clL( U, H^N)} \|\delta u\|_{ U_T},\notag
\end{align}
with~$C_H\coloneqq\|\Id\|_{\clL(H^N,(V^N)')}$.
By combining~\eqref{eq:deltabound},~\eqref{eq:deltap}, and~\eqref{eq:deltaxV}, we find, with $C_{\delta \bfA}\coloneqq \|\delta\bfA\|_{\clL(V^N,(V^N)')}$,
\begin{align}
\delta\clJ
&\le C_{\delta \bfA}^2 \, \frac1 N \mathfrak{C}_{\bfp}C_W\,\|\clE (x_{\sigma}+g)\|_{V_T^N} \,\|\clE p_{\sigma}\|_{V_T^N}
+ C_{\delta \bfA}\frac1 N \|\clE p_{\sigma}\|_{V_T^N}\|\delta\bfx\|_{V_T^N}\notag\\
&\quad+ C_{\delta \bfA}\,\frac1 N\mathfrak{C}_{\bfp}\,C_W\|\clE( x_{\sigma}+g)\|_{V_T^N}\,\left(\| Q_\rme\delta \bfx\|_{Y_T^N} +\| P_\rme\delta\bfx(T)\|_{Z^N}\right),\notag
\end{align}
that is, after multiplication by~$N$ and using Young inequalities,
\begin{align}
N\delta\clJ&\le C_{\delta \bfA}^2  \mathfrak{C}_{\bfp}C_W\|\clE (x_{\sigma}+g)\|_{V_T^N} \|\clE p_{\sigma}\|_{V_T^N}
+ C_{\delta \bfA}^2C_W \|\clE p_{\sigma}\|_{V_T^N} \|\clE(x_\sigma + g)\|_{V_T^N} \notag\\
&\quad+ C_{\delta \bfA}C_WC_H\|\clE p_{\sigma}\|_{V_T^N}\|\bfB\|_{\clL( U, H^N)} \|\delta u\|_{ U_T}\notag\\
&\quad+ C_{\delta \bfA}\mathfrak{C}_{\bfp}\,C_W\|\clE( x_{\sigma}+g)\|_{V_T^N}\,\left(\| Q_\rme\delta \bfx\|_{Y_T^N} +\|P_\rme\delta\bfx(T)\|_{Z^N}\right)\notag\\
&\le C_{\delta \bfA}^2 \, C_W (\mathfrak{C}_{\bfp}+1)\,\|\clE (x_{\sigma}+g)\|_{V_T^N} \,\|\clE p_{\sigma}\|_{V_T^N}\notag\\
&\quad+ \frac{1}{2N}C_{\delta \bfA}^2C_W^2C_H^2\|\clE p_{\sigma}\|_{V_T^N}^2\|\bfB\|_{\clL( U, H^N)}^2 +C_{\delta \bfA}^2\mathfrak{C}_{\bfp}^2\,C_W^2\|\clE( x_{\sigma}+g)\|_{V_T^N}^2\notag\\
&\quad+ \frac{N}{2}\|\delta u\|_{U_T}^2 +\frac12\| Q_\rme\delta \bfx\|_{Y_T^N}^2 +\frac12\|P_\rme\delta\bfx(T)\|_{Z^N}^2\notag.
\end{align}
Recalling~\eqref{eq:deltaJ}, we find
\begin{align}
\frac{N}2\delta\clJ
&\le C_{\delta \bfA}^2 \, C_W (\mathfrak{C}_{\bfp}+1)\,\|\clE (x_{\sigma}+g)\|_{V_T^N} \,\|\clE p_{\sigma}\|_{V_T^N}\notag\\
&\quad+ \frac{1}{2N}C_{\delta \bfA}^2C_W^2C_H^2\|\clE p_{\sigma}\|_{V_T^N}^2\|\bfB\|_{\clL( U, H^N)}^2 +C_{\delta \bfA}^2\mathfrak{C}_{\bfp}^2\,C_W^2\|\clE( x_{\sigma}+g)\|_{V_T^N}^2,\notag
\end{align}
therefore, for~$\clJ(\delta\bfx,\delta u)=\frac12\delta\clJ$ we find
\begin{align}
\frac12\delta\clJ
&\le \frac{1}{2N}C_{\delta \bfA}^2 \, C_W (\mathfrak{C}_{\bfp}+1)\,\|(\clE (x_{\sigma}+g),\clE p_{\sigma})\|_{V_T^N\times V_T^N}^2\notag\\
&\quad+ \frac{1}{2N^2}C_{\delta \bfA}^2C_W^2C_H^2\|\bfB\|_{\clL( U, H^N)}^2\|\clE p_{\sigma}\|_{V_T^N}^2 +\frac{1}{N}C_{\delta \bfA}^2\mathfrak{C}_{\bfp}^2\,C_W^2\|\clE( x_{\sigma}+g)\|_{V_T^N}^2\notag\\
&= \frac{1}{2}C_{\delta \bfA}^2 \, C_W (\mathfrak{C}_{\bfp}+1)\,\|(x_{\sigma}+g, p_{\sigma})\|_{V_T\times V_T}^2\notag\\
&\quad+ \frac{1}{2N}C_{\delta \bfA}^2C_W^2C_H^2\|\bfB\|_{\clL( U, H^N)}^2\|p_{\sigma}\|_{V_T}^2 +C_{\delta \bfA}^2\mathfrak{C}_{\bfp}^2\,C_W^2\|x_{\sigma}+g\|_{V_T}^2\notag\\
&\le  \frac{1}{2}C_{\delta \bfA}^2  C_W \left(1+\mathfrak{C}_{\bfp}+C_WC_H^2\|B\|_{\clL( U, H)}^2+2\mathfrak{C}_{\bfp}^2 C_W\right)
\|(x_{\sigma}+g, p_{\sigma})\|_{V_T\times V_T}^2,\notag
\end{align}
which ends the proof.
\end{proof}

Next, we compare the value of the optimal costs associated with the ensemble optimal control problem~\eqref{eq:OPextx} and the single parameter optimal  control problem~\eqref{eq:OP1sig}.

\begin{corollary}\label{coro:big-small1}
Let $(\bfx_\varSigma,u_{\varSigma})$ be the minimizer of~\eqref{eq:OPextx} and let $(x_{\sigma},u_{\sigma})$ be the minimizer of~\eqref{eq:OP1sig} with~$\bfx_\varSigma(0)=\clE x_{\sigma}(0)$. Then, there holds
\begin{align}\notag
0\le \clJ(\clE x_{\sigma}, u_{\sigma})-\clJ(\bfx_\varSigma,u_{\varSigma})
\le \|\bfA_{\varSigma} - \bfA_{\sigma}\|_{\clL(V^N,(V^N)')}  \fkC_2 \sqrt{3\fkC_1},
\end{align}
with~$\fkC_1$ as in~\eqref{eq:fkC} and with
\begin{align}\label{eq:Cbig-small}
\fkC_2 &\coloneqq  \clI(\bfx_\varSigma + \clE x_{\sigma},u_{\varSigma} + u_{\sigma}),\\
\mbox{where}\quad
\clI(z,v)&\coloneqq
\frac{1}{\sqrt2}\| v\|_{ U_T} +\frac{1}{\sqrt{2 N}}\|Q_\rme z \|_{ Y_T^N}+ \frac{1}{\sqrt{2 N}} \| P_\rme z(T)\|_{Z^N}.\label{eq:clI}
\end{align}
\end{corollary}

\begin{proof} The inequality~$0\le \clJ(\clE x_{\sigma}, u_{\sigma})-\clJ(\bfx_\varSigma,u_{\varSigma})$ holds because~$(\bfx_\varSigma,u_{\varSigma})$ minimizes~$\clJ$. To obtain the upper bound we estimate
\begin{align}
&\clJ_\delta\coloneqq\clJ(\clE x_{\sigma}, u_{\sigma})-\clJ(\bfx_\varSigma,u_{\varSigma})\notag\\
 &= \Bigl|\frac{1}{2} \langle u_{\varSigma} - u_{\sigma} , u_{\varSigma} + u_{\sigma}\rangle_{ U_T}+\frac{1}{2 N} \langle Q_\rme\bfx_\varSigma - Q_\rme\clE x_{\sigma},  Q_\rme\bfx_\varSigma + Q_\rme\clE x_{\sigma} \rangle_{ Y_T^N} \notag\\
&\quad + \frac{1}{2N}\langle P_\rme\bfx_\varSigma(T) - P_\rme\clE x_\sigma(T), P_\rme\bfx_\varSigma(T) + P_\rme\clE x_\sigma(T)\rangle_{Z^N}\Bigr|\notag \\
&\le \frac{1}{2} \| u_{\varSigma} - u_{\sigma} \|_{ U_T}\| u_{\varSigma} + u_{\sigma}\|_{ U_T}+\frac{1}{2 N} \| Q_\rme(\bfx_\varSigma - \clE x_{\sigma})\|_{ Y_T^N} \|Q_\rme(\bfx_{\varSigma} + \clE x_{\sigma}) \|_{ Y_T^N}\notag\\
&\quad + \frac{1}{2N} \| P_\rme(\bfx_\varSigma(T) - \clE x_\sigma(T))\|_{Z^N} \|P_\rme( \bfx_\varSigma(T) + \clE x_\sigma(T))\|_{Z^N}, \notag
\end{align}
hence, recalling~\eqref{eq:clI}, it follows that
\begin{align}
\clJ_\delta
&\le\clI(\bfx_\varSigma + \clE x_{\sigma},u_{\varSigma} + u_{\sigma})\, \clI(\bfx_\varSigma - \clE x_{\sigma},u_{\varSigma} - u_{\sigma})\notag\\
&\le \clI(\bfx_\varSigma + \clE x_{\sigma},u_{\varSigma} + u_{\sigma}) \sqrt{3\clJ(\bfx_\varSigma - \clE x_{\sigma},u_{\varSigma} - u_{\sigma})}.\notag
\end{align}
The claim follows from Lemma~\ref{lem:objdiff1}.
\end{proof}

\subsection{Cost of the proposed performant feedback control}

We compare the minimal cost associated with the minimizer~$(\bfx_\varSigma, u_\varSigma)$~of the extended system problem~\eqref{eq:OPextx} to the cost associated with the solution~$x_{\varSigma,\sigma}$ of~\eqref{eq:xsig-robFB} resulting from the feedback control $u_{\varSigma,\sigma}=  K_{\varSigma}(t,x_{\varSigma,\sigma}(t))$ given in \eqref{eq:robFB}.

The following result quantifies the difference of this feedback control associated to the single unknown parameter $\sigma$ compared to the optimal control associated to the ensemble of training parameters $\varSigma$ in the sense of the associated costs.
\begin{theorem}\label{thm:suboptimality} Assume that $x_0 \in V$,
let~$(\bfx_\varSigma, u_\varSigma)$ be the minimizer of~\eqref{eq:OPextx} with~$\bfx_\circ=\clE x_\circ$. Further, let~$\sigma \in \fkS$ and let~$x_{\varSigma,\sigma}$ be the solution of
\begin{align}\label{eq:cllsys}
 \dot{x}_{\varSigma,\sigma}(t) &=  \clA_\sigma  x_{\varSigma,\sigma}(t) + B  K_{\varSigma}(t,x_{\varSigma,\sigma}(t)) +  \clA_\sigma  g(t) - \dot{g}(t), \quad  x_{\varSigma,\sigma}(0)= x_\circ,
\end{align}
and let~$u_{\varSigma,\sigma}(t) \coloneqq   K_{\varSigma}(t,x_{\varSigma,\sigma}(t))$. Then, there holds
\begin{align}
&\hspace{-3em}0 \le \clJ(\clE x_{\varSigma,\sigma},u_{\varSigma,\sigma}) - \clJ(\bfx_\varSigma,u_\varSigma) \le \mathfrak{C}_{\rm subopt} \|\bfA_{\varSigma} - \bfA_{\sigma}\|_{\clL(\clD_\clA^N,H^N)},\notag\\
\mbox{with}\quad\mathfrak{C}_{\rm subopt}& \coloneqq \norm{\bfPi_\varSigma}{}\dnorm{X}{ H^N_T} (\dnorm{X}{(\clD_\clA^N)_T}+\dnorm{\clE g}{(\clD_\clA^N)_T}) \notag \\ &\quad+\dnorm{\bfh}{ H^N_T} (\dnorm{X}{(\clD_\clA^N)_T} + \|\clE g\|_{(\clD_\clA^N)_T}).
\label{eq:Csubopt}
\end{align}
\end{theorem}
\begin{proof} 
We commence by commenting on the  regularity of $\clE x_{\varSigma,\sigma}$ and $\bfh$ which will be used throughout the proof. Due to~\eqref{eq:DAsame}, we have that $\clD(\clA_\sigma^{\frac{1}{2}})=V$, for all~$\sigma\in\fkS$ (see, e.g.,~\cite[p.~183]{bensoussan2007representation}). Consequently, it follows by maximal regularity theory that $\bfh\in W_T(\mathcal{D}_\clA^N,H^N)$; see~\eqref{eq:def-Space_T} and~\cite[p.~187]{bensoussan2007representation}. Further, since $x_\circ \in V$, we have  $\clE x_{\varSigma,\sigma}\in W_T(\mathcal{D}_\clA^N,H^N)$ as well.

Next, recalling~\eqref{eq:mincost}, the minimal cost is given by
\begin{align}
\clJ(\bfx_\varSigma, u_\varSigma) &= \frac{1}{2} \langle \bfPi_\varSigma(T) \bfx_\circ,\bfx_\circ\rangle_{ H^N} + \langle \bfh(0), \bfx_\circ\rangle_{ H^N}\notag \\
&\quad+ \int_0^T \left(  \langle \bfh(t), \bff(t) \rangle_{ H^N} - \frac{1}{2}\|\bfB^\ast \bfh(t)\|_{U}^2\right) \mathrm dt.\label{eq:clJnfxSigma}
\end{align}
We observe that
\begin{align}
\langle \bfx_\circ, \bfPi_\varSigma(T) \bfx_\circ \rangle_{ H^N} =& \langle \clE x_{\varSigma,\sigma}(T), \bfPi_\varSigma(0) \clE x_{\varSigma,\sigma}(T) \rangle_{ H^N}-\fkT_1=\frac{1}{N}\dnorm{P_\rme \clE x_{\varSigma,\sigma}(T)}{Z^N}^2-\fkT_1\notag\\
\mbox{ with}\quad\fkT_1\coloneqq&  \int_0^T \frac{\mathrm d}{\mathrm dt} \langle \clE x_{\varSigma,\sigma}(t), \bfPi_{\varSigma}(T-t) \clE x_{\varSigma,\sigma}(t) \rangle_{ H^N} \mathrm dt\label{eq:cost-fkT1}
\end{align}
and,  denoting~$\underline\bfPi_\varSigma(t)\coloneqq\bfPi_\varSigma(T-t)$, we find
\begin{subequations}\label{eq:fkT1-1}
\begin{align}
\fkT_1 &=\fkT_{1,1}+\fkT_{1,2}+\fkT_{1,3},\\
\mbox{with}\quad\fkT_{1,1}&\coloneqq\int_0^T\langle \clE \dot x_{\varSigma,\sigma}(t), \underline\bfPi_\varSigma(t) \clE x_{\varSigma,\sigma}(t) \rangle_{ H^N} \mathrm dt \\
\fkT_{1,2}&\coloneqq\int_0^T\langle \clE x_{\varSigma,\sigma}(t), \underline\bfPi_\varSigma(t) \clE  \dot x_{\varSigma,\sigma}(t) \rangle_{ H^N} \mathrm dt \\
\fkT_{1,3}&\coloneqq\int_0^T\langle \clE x_{\varSigma,\sigma}(t), \underline{\dot\bfPi}_\varSigma(t) \clE x_{\varSigma,\sigma}(t) \rangle_{ H^N} \mathrm dt.
\end{align}
\end{subequations}
Using~$\clE \dot x_{\varSigma,\sigma}=\bfA_\sigma x_{\varSigma,\sigma} - \bfB\bfB^* (\underline\bfPi_\varSigma\clE x_{\varSigma,\sigma}+\bfh) + \bff_\sigma$, where we have abbreviated
\begin{equation}\label{eq:bffsig}
\bff_\sigma(t) \coloneqq \bfA_\sigma \clE g(t) - \clE \dot{g}(t),
\end{equation}
we obtain
\begin{subequations}\label{eq:fkT23}
\begin{align}
&\fkT_{1,1}+\fkT_{1,2}=\fkT_2+\fkT_3,\\
&\mbox{with}\quad\fkT_2\coloneqq \int_0^T \langle (\bfA_\sigma - \bfB\bfB^* \underline\bfPi_\varSigma)\clE  x_{\varSigma,\sigma}(t) ,\underline\bfPi_\varSigma(t)  \clE  x_{\varSigma,\sigma}(t)   \rangle_{ H^N}\mathrm dt \notag\\
&\hspace{5em}+\int_0^T\langle\clE  x_{\varSigma,\sigma}(t) , \underline\bfPi_\varSigma(t)(\bfA_\sigma
-\bfB\bfB^* \underline\bfPi_\varSigma(t) )  \clE  x_{\varSigma,\sigma}(t)   \rangle_{ H^N}\mathrm dt \\
&\mbox{and}\quad\fkT_3\coloneqq\int_0^T \langle \bff_\sigma(t)-\bfB\bfB^*\bfh(t) ,\underline\bfPi_\varSigma(t)  \clE  x_{\varSigma,\sigma}(t)   \rangle_{ H^N}\mathrm dt \notag\\
&\quad\hspace{5em}+\int_0^T\langle\clE  x_{\varSigma,\sigma}(t) , \underline\bfPi_\varSigma(t)(\bff_\sigma(t)-\bfB\bfB^*\bfh(t)) \rangle_{ H^N}\mathrm dt
\end{align}
\end{subequations}

With~$X\coloneqq \clE x_{\varSigma,\sigma}$ and~$\delta\bfA= \bfA_\varSigma- \bfA_\sigma$, we find
\begin{align}
\fkT_2&=  -\int_0^T \langle X(t) ,(\delta\bfA^*\underline\bfPi_\varSigma+\underline\bfPi_\varSigma \delta\bfA)  X(t)   \rangle_{ H^N}\mathrm dt\notag\\
&\quad+\int_0^T \langle X(t) ,(\bfA_\varSigma^*\underline\bfPi_\varSigma+\underline\bfPi_\varSigma \bfA_\varSigma-2\underline\bfPi_\varSigma\bfB\bfB^*\underline\bfPi_\varSigma)  X(t)   \rangle_{ H^N}\mathrm dt \notag
\end{align}
and, recalling~\eqref{eq:extRiccati},
\begin{align}
\fkT_2&= -\int_0^T ( \langle X(t) ,(\delta\bfA^*\underline\bfPi_\varSigma+\underline\bfPi_\varSigma \delta\bfA)  X(t)   \rangle_{ H^N} -\langle X(t) , \underline{\dot\bfPi}_\varSigma(t) X(t)  \rangle_{ H^N} ) \mathrm dt\notag\\
&\quad-\int_0^T \langle X(t) ,(\underline\bfPi_\varSigma\bfB\bfB^*\underline\bfPi_\varSigma+\frac1N Q_\rme^*Q_\rme)  X(t)   \rangle_{ H^N}\mathrm dt \notag
\end{align}
and, from~\eqref{eq:fkT1-1} and~\eqref{eq:fkT23},
\begin{align}
\fkT_1 =\fkT_{3}+ \fkT_{2}+\fkT_{1,3} =\fkT_3&-\int_0^T \langle X(t) ,(\delta\bfA^*\underline\bfPi_\varSigma+\underline\bfPi_\varSigma \delta\bfA)  X(t)   \rangle_{ H^N}\mathrm dt\notag\\
&-2\clJ(X,v_X)+\frac1N \dnorm{P_\rme  X(T)}{Z^N}^2\notag.
\end{align}
with~$v_X\coloneqq-\bfB^*\underline\bfPi_\varSigma X(t)$. Next, recalling~\eqref{eq:clJnfxSigma} and~\eqref{eq:cost-fkT1},
\begin{align}
2\clJ(\bfx_\varSigma, u_\varSigma) &=\frac{1}{N}\dnorm{P_\rme X(T)}{Z^N}^2-\fkT_1 \notag \\
&\quad+ 2\langle \bfh(0), \bfx_\circ\rangle_{ H^N}+ \int_0^T \left(2  \langle \bfh(t), \bff(t) \rangle_{ H^N} - \|\bfB^\ast \bfh(t)\|_{U}^2\right) \mathrm dt\notag\\
&=-\fkT_{3}+\int_0^T \langle X(t) ,(\delta\bfA^*\underline\bfPi_\varSigma+\underline\bfPi_\varSigma \delta\bfA)  X(t)   \rangle_{ H^N}\mathrm dt+2\clJ(X,v_X)\notag \\
&\quad+ 2\langle \bfh(0), \bfx_\circ\rangle_{ H^N}+ \int_0^T \left(2  \langle \bfh(t), \bff(t) \rangle_{ H^N} - \|\bfB^\ast \bfh(t)\|_{U}^2\right) \mathrm dt.\notag
\end{align}

Hence, with~$u_{\varSigma,\sigma}=u_X\coloneqq -\bfB^*(\underline\bfPi_\varSigma X+\bfh)=v_X-\bfB^*\bfh$, we use
\begin{align}
\dnorm{u_X}{U}^2=\dnorm{v_X}{U}^2-2(v_X,\bfB^*\bfh)_U+\dnorm{\bfB^*\bfh}{U}^2\notag
\end{align}
to obtain
\begin{align}
2\delta\clJ_{\varSigma,\sigma;\varSigma}&\coloneqq2\clJ( \clE x_{\varSigma,\sigma},u_{\varSigma,\sigma})-2\clJ(\bfx_\varSigma, u_\varSigma) =2\clJ(X,u_X)-2\clJ(\bfx_\varSigma, u_\varSigma) \notag\\
&=\fkT_{3}-\int_0^T \langle X(t) ,(\delta\bfA^*\underline\bfPi_\varSigma+\underline\bfPi_\varSigma \delta\bfA)  X(t)   \rangle_{ H^N}\mathrm dt - 2\langle \bfh(0), \bfx_\circ\rangle_{ H^N}\notag\\
&\quad-2\int_0^T \left( \langle \bfh(s), \bff(s) \rangle_{ H^N}+(v_X,\bfB^*\bfh)_ U-\dnorm{\bfB^*\bfh(t)}{ U}^2\right) \mathrm ds.\notag
\end{align}

For~$\fkT_3$ as in~\eqref{eq:fkT23}, we find, using the symmetry of~$\underline\bfPi_\varSigma$,
\begin{align}
\quad\fkT_3&=2\int_0^T \langle \bff_\sigma(t),\underline\bfPi_\varSigma(t) X(t)   \rangle_{ H^N}
\mathrm dt+2\int_0^T \langle \bfB^*\bfh(t),v_X(t)   \rangle_{ H^N}
\mathrm dt \notag
\end{align}
which leads to
\begin{align}
2\delta\clJ_{\varSigma,\sigma;\varSigma}&=-\int_0^T \langle X(t) ,(\delta\bfA^*\underline\bfPi_\varSigma+\underline\bfPi_\varSigma \delta\bfA)  X(t)   \rangle_{ H^N}\mathrm dt- 2\langle \bfh(0), \bfx_\circ\rangle_{ H^N}\label{eq:difJSigsig-Sig.1}\\
&\quad+ 2\int_0^T \left(\langle \bff_\sigma(t),\underline\bfPi_\varSigma(t) X(t)   \rangle_{ H^N}- \langle \bfh(t), \bff(t) \rangle_{ H^N}+\dnorm{\bfB^*\bfh(t)}{ U}^2\right) \mathrm dt.\notag
\end{align}

Recalling system~\eqref{eq:h}, satisfied by~$\bfh$, we find
\begin{subequations}\label{eq:h0X0}
\begin{align}
&- \langle \bfh(0), \bfx_\circ \rangle_{ H^N}=\int_0^T\frac{\rmd}{\rmd t} \langle \bfh(t) ,X(t)\rangle_{ H^N} \mathrm dt=\fkT_4+\fkT_5\\
\mbox{with}&\quad\fkT_4\coloneqq\int_0^T \langle \dot \bfh(t) ,X(t)\rangle_{ H^N}\mathrm dt\quad\mbox{and}\quad\fkT_5\coloneqq\int_0^T\langle \bfh(t) ,\dot X(t)\rangle_{ H^N} \mathrm dt.
\end{align}
\end{subequations}

We observe that
\begin{align}
\fkT_4&=-\int_0^T \langle \left(\bfA_\varSigma^\ast - \underline\bfPi_\varSigma(t)\bfB\bfB^\ast\right) \bfh(t) + \underline \bfPi_\varSigma(t)\bff(t) ,X(t)\rangle_{ H^N}\mathrm dt,\notag\\
\fkT_5&=\int_0^T\langle \bfh(t) ,(\bfA_\sigma  - \bfB\bfB^* (\underline\bfPi_\varSigma (t)X(t)+\bfh)) + \bff_\sigma\rangle_{ H^N} \mathrm dt\notag\\
&=\int_0^T\langle (\bfA_\sigma^\ast  - \underline\bfPi_\varSigma (t)\bfB\bfB^*)\bfh (t) ,X(t)\rangle_{ H^N} +\langle\bfh(t), \bff_\sigma(t)\rangle_{ H^N}-\dnorm{\bfB^*\bfh (t)}{U}^2 \mathrm dt,\notag
\end{align}
and
\begin{align}
\fkT_4+\fkT_5&=-\int_0^T\left(\langle \delta\bfA^\ast \bfh (t) ,X(t)\rangle_{ H^N} \mathrm dt-\langle\bfh(t), \bff_\sigma(t)\rangle_{ H^N}+\dnorm{\bfB^*\bfh (t)}{U}^2\right) \mathrm dt \notag\\
&\quad-\int_0^T\langle  \underline \bfPi_\varSigma(t)\bff(t) ,X(t)\rangle_{ H^N}\mathrm dt.\notag
\end{align}
This relation combined with~\eqref{eq:difJSigsig-Sig.1} and~\eqref{eq:h0X0} gives, with~$\delta\bff=\bff-\bff_\sigma$,
\begin{align}
2\delta\clJ_{\varSigma,\sigma;\varSigma}&=-\int_0^T \langle X(t) ,(\delta\bfA^*\underline\bfPi_\varSigma+\underline\bfPi_\varSigma \delta\bfA)  X(t)   \rangle_{ H^N}\mathrm dt + 2(\fkT_4+\fkT_5)\notag \\&\quad+ 2\int_0^T \left(\langle \bff_\sigma(t),\underline\bfPi_\varSigma(t) X(t)   \rangle_{ H^N}- \langle \bfh(t), \bff(t) \rangle_{ H^N} + \|\bfB^\ast \bfh(t)\|_U^2\right) \mathrm dt
\notag\\
&=-\int_0^T\! \langle X(t) ,(\delta\bfA^*\underline\bfPi_\varSigma+\underline\bfPi_\varSigma \delta\bfA)  X(t)   \rangle_{ H^N}+2\langle \delta\bfA^\ast \bfh (t) ,X(t)\rangle_{ H^N} \mathrm dt\notag\\
&\quad-  2\int_0^T \left(\delta\bff(t),\underline\bfPi_\varSigma(t) X(t)   \rangle_{ H^N}+ \langle \bfh(t), \delta\bff(t) \rangle_{ H^N}\right) \mathrm dt
\notag
\end{align}

From~\eqref{eq:bff}  and~\eqref{eq:bffsig},
we have that~$\delta\bff= \delta\bfA_\sigma \clE g(t)$, hence
\begin{align}
2\delta\clJ_{\varSigma,\sigma;\varSigma}&=-2\int_0^T \langle X(t) ,\underline\bfPi_\varSigma \delta\bfA   X(t)   \rangle_{ H^N}+2\langle  \bfh (t) ,\delta\bfA X(t)\rangle_{ H^N} \mathrm dt\notag\\
&\quad-  2\int_0^T\langle\delta\bfA \clE g(t),\underline\bfPi_\varSigma(t) X(t)   -\bfh(t)\rangle_{ H^N} \mathrm dt.\notag
\end{align}
Finally, denoting~$\norm{\bfPi_\varSigma}{}\coloneqq\dnorm{\bfPi_\varSigma}{L^\infty(0,T;\clL( H^N))}$ and by recalling the notation for~$\clX_T$ in~\eqref{eq:def-Space_T}, we obtain the estimate
\begin{align}
\delta\clJ_{\varSigma,\sigma;\varSigma}&\le \mathfrak{C}_{\rm subopt} \dnorm{\delta\bfA}{\clL(\clD_\clA^N, H^N)}\notag
\end{align}
with~$\mathfrak{C}_{\rm subopt}\coloneqq \norm{\bfPi_\varSigma}{}\dnorm{X}{ H^N_T} (\dnorm{X}{(\clD_\clA^N)_T} + \dnorm{\clE g}{(\clD_\clA^N)_T})+\dnorm{\bfh}{ H^N_T} (\dnorm{X}{(\clD_\clA^N)_T} + \|\clE g\|_{(\clD_\clA^N)_T})$.
\end{proof}

The following result is an immediate consequence of Corollary~\ref{coro:big-small1} and Theorem~\ref{thm:suboptimality}.
\begin{corollary}\label{coro:smallrob}
Given a parameter $\sigma \in \fkS$, let $(x_\sigma,u_\sigma)$ be the minimizer of ~\eqref{eq:OP1sig} and let~$\clE x_{\varSigma,\sigma}$ the solution of~\eqref{eq:cllsys} with $u_{\varSigma,\sigma}(t) =  K_{\varSigma}(t,x_{\varSigma,\sigma}(t))$. Then, there holds
\begin{align}
&|\clJ(\clE x_\sigma , u_\sigma) - \clJ (\clE x_{\varSigma,\sigma}, u_{\varSigma,\sigma} )|\le \fkC_{2}\sqrt{3\mathfrak{C}_{1}} \|\delta\bfA\|_{\clL(V^N,(V^N)')} + \mathfrak{C}_{\rm subopt} \|\delta\bfA\|_{\clL(\clD_\clA^N,H^N)},\notag
\end{align}
with~$\delta\bfA=\bfA_{\varSigma} - \bfA_\sigma$, and with the constants $\fkC_{1}$, $\fkC_2$,  $\fkC_{\rm subopt}$, as in~\eqref{eq:fkC}, \eqref{eq:Cbig-small}, \eqref{eq:Csubopt}.
\end{corollary}

\subsection{Remarks}
Within the statement on Corollary~\ref{coro:smallrob} we use two operator norms for the difference~$\delta\bfA$, namely, $\|\delta\bfA\|_{\clL(V^N,(V^N)')}$ and~$\|\delta\bfA\|_{\clL(\clD_\clA^N,H^N)}$. We may wonder whether these norms are equivalent in the intersection space~$\clL(V^N,(V^N)')\bigcap\clL(\clD_\clA^N,H^N)$. We can provide a condition which ensures equivalence as follows.

Let  $\sigma\in \fkS$ be as in Corollary~\ref{coro:smallrob}
and assume that $\clA_\sigma$ commutes which each operator with index in the training set $\varSigma$.
Define $\underline\clA_\rho\coloneqq \clA_\sigma - 2\rho I$, with $\rho$ as in \eqref{eq:aux1}. This operator generates an analytic semigroup satisfying $\|\exp (\underline\clA_\rho t)\|_{\clL(H)}\le \exp(-\rho t)$, thus~$\underline\clA_\rho$ is an operator of type $(\omega,M)$ for some $\omega <\frac{\pi}{2}$ and $M>0$,  and the fractional power $(-\underline\clA_\rho )^{-\frac{1}{2}}$ can be expressed as contour integral in the resolvent set of  $(-\underline\clA_\rho )$; see~\cite[p.~167]{bensoussan2007representation}. From here it follows that $(-\underline\clA_\rho )^{-\frac{1}{2}}$ commutes with every $\clA_\sigma$ with $\sigma \in \varSigma$.

As a second preliminary for the following computation we recall that
since $\clD(\underline\clA_\rho) = \clD(\underline\clA_\rho^*)$ the mapping $(-\underline\clA_\rho)^{\frac{1}{2}}$ is an isomorphism between ${H}$ and $V'$ as well as between $\clD(\underline\clA_\rho)$ and $V$. Thus, there exist constants $C_i, \, i\in \{1,2\}$ such that
\begin{align}
\frac{1}{C_1} \|w\|_H \le \| (-\underline\clA_\rho)^{\frac{1}{2}}w\|_{V'} \le C_1\|w\|_H,  \quad\text{and}\quad
\frac{1}{C_2} \|w\|_{\clD_\clA} \le \| (-\underline\clA_\rho)^{\frac{1}{2}}w\|_V \le C_2\|w\|_{\clD_\clA}.\notag
\end{align}
For~$v\in V$ and $\varsigma \in \varSigma \cup \{\sigma\}$ we have
\begin{align}
\|\clA_\varsigma v\|_{V'} & \le C_1 \|(-\underline\clA_\rho)^{-\frac{1}{2}} \clA_\varsigma v\|_H = C_1 \|\clA_\varsigma (-\underline\clA_\rho)^{-\frac{1}{2}} v\|_H\notag \\& \le C_1 \|\clA_\varsigma\|_{\clL(\clD_\clA,H)} \|(-\underline\clA_\rho)^{-\frac{1}{2}}  v\|_{\clD_\clA}  \le C_1 C_2 \|\clA_\varsigma\|_{\clL(\clD_\clA,H)} \|v\|_V,\notag
\end{align}
which implies that $\|\clA_\varsigma\|_{\clL(V,V')} \le C_1C_2 \|\clA_\varsigma\|_{\clL(\clD_\clA,H)}$.

On the other hand, for~$v\in\clD_\clA$, there holds
\begin{align}
\|\clA_\varsigma v\|_{H} & \le C_1 \|(-\underline\clA_\rho)^{\frac{1}{2}} \clA_\varsigma v\|_{V'} = C_1 \|\clA_\varsigma (-\underline\clA_\rho)^{\frac{1}{2}} v\|_{V'} \notag\\&\le C_1 \|\clA_\varsigma\|_{\clL(V,V')} \|(-\underline\clA_\rho)^{\frac{1}{2}}  v\|_{V} \le  C_1 C_2 \|\clA_\varsigma\|_{\clL(V,V')} \|v\|_{\clD_\clA},\notag
\end{align}
which gives~$\|\clA_\varsigma\|_{\clL(\clD_\clA,H)} \le C_1C_2 \|\clA_\varsigma\|_{\clL(V,V')}$. Hence, for~$\delta\bfA_\varsigma=\bfA_{\varSigma} - \bfA_\varsigma$, we find
\begin{align}
\frac{1}{C_1C_2} \|\delta\bfA_\varsigma\|_{\clL(\clD_\clA^N,H^N)} \le \|\delta\bfA_\varsigma\|_{\clL(V^N,(V^N)')} \le C_1 C_2 \|\delta\bfA_\varsigma\|_{\clL(\clD_\clA^N,H^N)}.\notag
\end{align}

\section{Comparing state trajectories}\label{C:comp-traj}

We compare the state trajectories associated  to problems~\eqref{eq:OPextx} and~\eqref{eq:OP1sig} and to the system~\eqref{eq:cllsys} under the proposed feedback~\eqref{eq:robFB}.

\subsection{Trajectories associated to~\eqref{eq:OPextx} and~\eqref{eq:cllsys}}
We have the following result.
\begin{theorem}\label{thm:pointwise}
Let~$x_\circ \in V$, let $(\bfx_\varSigma, u_\varSigma)$ be the minimizer of problem~\eqref{eq:OPextx} and let~$x_{\varSigma,\sigma}$ be the solution of~\eqref{eq:cllsys} with~$u_{\varSigma,\sigma}(t) =  K_{\varSigma}(t,x_{\varSigma,\sigma}(t))$. Then, for~$t \in [0,T]$, we have
\begin{align}
\|\bfx_\varSigma(t) - \clE x_{\varSigma,\sigma}(t)\|_{H^N} \leq \|\bfA_\varSigma - \bfA_\sigma\|_{\clL(\clD_\clA^N,H^N)}\mathfrak{C}_T \| \clE x_{\varSigma,\sigma} + \clE g \|_{L^2(0,T;\clD_\clA^N)},
\end{align}
where~$\mathfrak{C}_T = \sqrt{T} \max{(1, e^{-\rho T})} e^{ T (\rho +\mathfrak{C}_{\mathrm{uni}})}$ with~$\mathfrak{C}_{\mathrm{uni}} \coloneqq\| \bfB \bfB^\ast \bfPi_{\varSigma}\|_{L^\infty(0,T;\mathcal{L}(H^N))}$.
\end{theorem}

\begin{proof}
With~$W_T(\clX,\clY)$ as in~\eqref{eq:def-Space_T}, we have~$\clE x_{\varSigma,\sigma}\in W_T(\clD_\clA^N,H^N)$, due to $x_\circ \in V$. Denoting $\delta \bfx \coloneqq \bfx_\varSigma - \clE x_{\varSigma,\sigma}$, $\delta \bfA \coloneqq \bfA_\varSigma - \bfA_\sigma$, and $\bff_\sigma \coloneqq \bfA_\sigma \clE g - \clE \dot{g}$, there holds
\begin{align}
\dot{\delta \bfx}(t) = \bfA_\varSigma \delta \bfx(t) + \delta \bfA \clE x_{\varSigma,\sigma}(t) + \bff(t) - \bff_\sigma(t) + \bfB (u_\varSigma -  K_{\varSigma}(t,x_{\varSigma,\sigma}(t))),\notag
\end{align}
for $t>0$ and $\delta \bfx(0) = 0$. Using \eqref{eq:uopt} and~\eqref{eq:robFB}, we obtain
\begin{align}\label{eq:fbdiff}
u_\varSigma(t) -  K_{\varSigma}(t,x_{\varSigma,\sigma}(t)) = - \bfB^\ast \bfPi_\varSigma(T-t) \delta \bfx(t).
\end{align}
Thus, recalling~\eqref{eq:bff} and~\eqref{eq:bffsig}, we have that~$\bff(t) - \bff_\sigma(t) = \delta\bfA \clE g(t)$, and we obtain
\begin{align}
\dot{\delta \bfx}(t) = \left(\bfA_\varSigma -  \bfB \bfB^* \bfPi_\varSigma(T-t)\right) \delta \bfx(t) + \delta \bfA \clE x_{\varSigma,\sigma}(t) + \delta\bfA \clE g(t).\notag
\end{align}
We can represent the solution (see, e.g.,~\cite[Prop.~3.4, Part II, Ch.~1]{bensoussan2007representation}) as
\begin{align}
\delta \bfx(t) = \int_0^t \bfS_\varSigma(t-s)  \left( - \bfB \bfB^* \bfPi_\varSigma(T-s) \delta \bfx(s) + \delta \bfA \left( \clE x_{\varSigma,\sigma}(s) + \clE g(s) \right) \right)\mathrm ds,\notag
\end{align}
leading to the estimate
\begin{align}
\|\delta \bfx(t)\|_{H^N} &\leq \int_0^t \left(\|\bfS_\varSigma(t-s)\|_{\clL(H^N)} \| \bfB \bfB^* \bfPi_\varSigma(T-s)\|_{\clL(H^N)} \|\delta \bfx(s)\|_{ H^N} \right)\mathrm ds\notag\\
&+ \int_0^t \left( \|\bfS_\varSigma(t-s)\|_{\clL(H^N)} \|\delta \bfA\|_{\clL(\clD_\clA^N,H^N)} \| \clE x_{\varSigma,\sigma}(s) + \clE g(s) \|_{\clD_\clA^N} \right) \mathrm ds\notag\\
&\leq \int_0^t \left(\mathfrak{C}_{\mathrm{uni}} \rme^{\rho (t-s)} \|\delta \bfx(s)\|_{H^N} \right)\mathrm ds\notag\\
&+ \int_0^t \left( \rme^{\rho (t-s)} \|\delta \bfA\|_{\clL(\clD_\clA^N, H^N)} \| \clE x_{\varSigma,\sigma}(s) + \clE g(s) \|_{\clD_\clA^N} \right) \mathrm ds\notag
\end{align}
where we took~$\mathfrak{C}_{\mathrm{uni}}\coloneqq\| \bfB \bfB^* \bfPi_\varSigma\|_{L^\infty(0,T;\clL(H^N)}$ and used $\|\bfS_\varSigma(t-s)\|_{\clL(H^N)} \le \rme^{\rho (t-s)}$. Multiplication by~$\rme^{-\rho t}$ gives
\begin{align}
\|\rme^{-\rho t}\delta \bfx(t)\|_{H^N} &\leq \int_0^t \left(\mathfrak{C}_{\mathrm{uni}}  \|\rme^{-\rho s} \delta \bfx(s)\|_{H^N} \right)\mathrm ds\notag\\
&\quad+ \|\delta \bfA\|_{\clL(\clD_\clA^N,H^N)}  \int_0^t \left( \rme^{-\rho s} \| \clE x_{\varSigma,\sigma}(s) + \clE g(s) \|_{\clD_\clA^N} \right) \mathrm ds.\notag
\end{align}
Then, Gronwall's lemma gives
\begin{align}
\|\rme^{-\rho t} \delta \bfx(t)\|_{H^N} &\leq  \rme^{\int_0^t  \mathfrak{C}_{\mathrm{uni}}  \mathrm ds}\|\delta \bfA\|_{\clL(\clD_\clA^N,H^N)} \int_0^t \left(\rme^{-\rho s}\| \clE x_{\varSigma,\sigma}(s) + \clE g(s) \|_{\clD_\clA^N} \right)\mathrm ds,\notag
\end{align}
and finally we obtain
\begin{align}
\|\delta \bfx(t)\|_{H^N} &\leq \left( \|\delta \bfA\|_{\clL(\clD_\clA^N, H^N)} \int_0^t \left(\rme^{-\rho s}\| \clE x_{\varSigma,\sigma}(s) + \clE g(s) \|_{\clD_\clA^N} \right)\mathrm ds \right) \rme^{ t (\rho +\mathfrak{C}_{\mathrm{uni}} )}\notag\\
&\leq \|\delta \bfA\|_{\clL(\clD_\clA^N, H^N)} \left(\int_0^t \rme^{-2\rho s} \mathrm ds\right)^{\frac12} \| \clE x_{\varSigma,\sigma} + \clE g \|_{L^2(0,t;\clD_\clA^N)}  \rme^{ t (\rho +\mathfrak{C}_{\mathrm{uni}} )}\notag\\
&\leq \|\delta \bfA\|_{\clL(\clD_\clA^N, H^N)}\sqrt{T} \max{(1, \rme^{-\rho T})} \| \clE x_{\varSigma,\sigma} + \clE g \|_{L^2(0,T;\clD_\clA^N)} \rme^{ T (\rho +\mathfrak{C}_{\mathrm{uni}} )}\!,\notag
\end{align}
where we used the Cauchy--Schwarz inequality in the second step.
\end{proof}

A close inspection of~\eqref{eq:fbdiff} reveals the following bound on the difference of the controls.
\begin{corollary}\label{coro:controlpointwise1}
Under the assumptions of Theorem~\ref{thm:pointwise}, there holds, for~$t\in [0,T]$, that
\begin{align}
\|u_\varSigma(t) - u_{\varSigma,\sigma}(t)\|_{ U} 
\le \|\bfA_\varSigma - \bfA_\sigma\|_{\clL(\clD_\clA^N, H^N)}  \mathfrak{C}_{\mathrm{uni}2} \mathfrak{C}_T \| \clE x_{\varSigma,\sigma} + \clE g \|_{L^2(0,T;\clD_\clA^N)}.\notag
\end{align}
with the uniform bound $\mathfrak{C}_{\mathrm{uni}2}\coloneqq \|\bfB^* \bfPi_{\varSigma}\|_{L^\infty(0,T;\mathcal{L}(H^N,U))}$, and~$\mathfrak{C}_T$ as in Theorem~\ref{thm:pointwise}.
\end{corollary}

\subsection{Trajectories associated to~\eqref{eq:OPextx} and~\eqref{eq:OP1sig}}
We have the following result.

\begin{theorem}\label{thm:pointwisestate2}
Let~$(\bfx_\varSigma, u_\varSigma)$ be the minimizer of~\eqref{eq:OPextx} and let~$(x_\sigma, u_\sigma)$ be the minimizer of~\eqref{eq:OP1sig}. Then, we have
\begin{align}
\|\bfx_\varSigma(t) - \clE x_\sigma(t)\|_{H^N} \le \max{\left(1,e^{(\frac12+\rho)T}\right)}\fkC_{T2}\|\bfA_\varSigma - \bfA_\sigma\|_{\clL(V^N,(V^N)')},\notag
\end{align}
with~$\fkC_{T2}\coloneqq   \left(\theta^{-\frac{1}{2}} \|\clE (x_\sigma + g)\|_{L^2(0,T;V^N)} +  \sqrt{2\mathfrak{C}} \|\bfB\|_{\clL( U,H^N)}\right)$,
where~$\rho$ and~$\theta$ are as in~\eqref{eq:aux1}, and~$\mathfrak{C}_1$ is as in~\eqref{eq:fkC}.
\end{theorem}
\begin{proof}
Let now~$\delta \bfx\coloneqq \bfx_\varSigma - \clE x_\sigma$, $\delta u \coloneqq u_\varSigma - u_\sigma$, and $\delta\bfA = \bfA_\varSigma - \bfA_\sigma$. Recall that $\delta \bfx$ satisfies~\eqref{eq:optcondiffA}. Thus, with~$\theta$ and~$\rho$ as in~\eqref{eq:aux1}, we find
\begin{align}
&\frac12 \frac{\mathrm d}{\mathrm dt} \|\delta \bfx_\varSigma\|_{H^N}^2 \notag\\&\!\!\quad= \langle \bfA_\varSigma \delta \bfx_\varSigma, \delta \bfx_\varSigma\rangle_{H^N} + \langle \delta \bfA \clE (x_\sigma(t) + g(t)), \delta \bfx_\varSigma \rangle_{(V^N)',V^N} + \langle \bfB \delta u_\varSigma, \delta \bfx_\varSigma \rangle_{H^N}\notag\\
&\!\!\quad\le \rho \|\delta \bfx_\varSigma\|_{H^N}^2 - \theta \|\delta \bfx_\varSigma\|_{V^N}^2 + \|\delta \bfA\|_{\clL(V^N,(V^N)')}^2 \frac{1}{2\theta} \|\clE(x_\sigma(t) + g(t))\|_{V^N}^2\notag\\
&\!\!\quad\quad  + \frac{\theta}{2} \|\delta \bfx_\varSigma\|_{V^N}^2 + \|\bfB\|_{\clL(U,H^N)}^2 \frac12 \|\delta u_\varSigma\|^2_{U}+ \frac12 \|\delta \bfx_\varSigma\|_{H^N}^2,\notag
\end{align}
where we used Young's inequality $ab \le \frac{a^2}{2\theta} + \frac{\theta b^2}{2}$ for $a\ge0, b\ge 0$. Furthermore, with Gronwall's lemma and~$\delta \bfx_\varSigma(0) = 0$, we obtain
\begin{align}
\|\delta \bfx_\varSigma(t)\|_{H^N}^2 &\le \int_0^t e^{(1+2\rho)(t-s)} \left( \|\delta \bfA\|_{\clL(V^N,(V^N)')}^2 \frac{1}{\theta}\|\clE(x_\sigma(s) + g(s))\|_{V^N}^2 \right) \mathrm ds\notag\\
&\quad+\int_0^t e^{(1+2\rho)(t-s)} \left(  \|\bfB\|_{\clL(U,H^N)}^2 \|\delta u_\varSigma(s)\|^2_{U}\right) \mathrm ds\notag\\
&\le  \|\delta \bfA\|_{\clL(V^N,(V^N)')}^2 \max{\left(1,e^{(1+2\rho)T}\right)}\int_0^t \frac{1}{\theta} \|\clE(x_\sigma(s) + g(s))\|_{V^N}^2 \mathrm ds\notag\\
&\quad+  \|\bfB\|_{\clL(U,H^N)}^2 \max{\left(1,e^{(1+2\rho)T}\right)}\int_0^t \|\delta u_\varSigma(s)\|^2_{U} \mathrm ds.\notag
\end{align}
Furthermore, with~$\mathfrak{C}_{1}$ as in Lemma~\ref{lem:objdiff1} we have $\int_0^t \|\delta u_\varSigma(s)\|_{U}^2 \mathrm ds \le 2\clJ(\delta\bfx, \delta u) \le 2\mathfrak{C}_1\|\delta\bfA\|^2_{\clL(V^N,(V^N)')}$. Finally, we conclude that, for all~$t\in[0,T]$,
\begin{align}
\|\delta \bfx_\varSigma(t)\|_{H^N} \le \max{\left(1,e^{(\frac12+\rho)T}\right)}\fkC_{T2}\|\delta \bfA\|_{\clL(V^N,(V^N)')} .\notag
\end{align}
with~$\fkC_{T2}\coloneqq   \left( \theta^{-\frac{1}{2}} \|\clE (x_\sigma + g)\|_{L^2(0,T;V^N)} +  \sqrt{2\mathfrak{C}_{1}}\|\bfB\|_{\clL( U,H^N)}\right)$.
\end{proof}

A similar estimate holds for the difference~$\delta u$ of the corresponding controls, as follows.
\begin{corollary}\label{coro:controlpointwise2}
Under the conditions of Theorem~\ref{thm:pointwisestate2}, there holds
\begin{align}
&\|u_\varSigma(t) - u_\sigma(t)\|_{U}\le\|B\|_{\clL(U,H)}\mathfrak{C}_{T3}\|\bfA_\varSigma - \bfA_\sigma\|_{\clL(V^N,(V^N)')}\notag
\end{align}
with~$\fkC_{T3}\coloneqq \max{(1,\rme^{\frac{T}{2}\left(\|Q_\rme\|_{\clL(H^N,Y^N)}^2+2\rho\right)})}\sqrt{2\mathfrak{C}_{1} +\theta^{-1}\|p_\sigma(s)\|_{V_T}^2}$, where~$\rho\in\bbR$ and~$\theta>0$ are as in~\eqref{eq:aux1}, and~$\mathfrak{C}_{1}$ is as in~\eqref{eq:fkC}.
\end{corollary}
\begin{proof}
Let~$\delta \bfx = \bfx_\varSigma - \clE x_\sigma$, $\delta u = u_\varSigma - u_\sigma$, and $\delta\bfA = \bfA_\varSigma - \bfA_\sigma$; further, let~$\delta \bfp = \bfp_\varSigma - \clE p_\sigma$. From~\eqref{eq:optcondiffC} we find,  for~$t\in(0,T)$,
\begin{align}
\frac{\mathrm d}{\mathrm dt} \|\delta \bfp(T-t)\|_{H^N}^2 &=2(\bfA^\ast_{\varSigma} \delta\bfp(T-t) ,\delta\bfp(T-t))_{H^N} \notag\\
&+2(\delta\bfA^\ast \frac1 N \clE p_{\sigma}(T-t) + \frac1 N Q_\rme^*Q_\rme \delta \bfx(T-t),\delta\bfp(T-t))_{H^N}.\notag
\end{align}
Then, by~\eqref{eq:aux1}, \eqref{eq:bilinearform},
\begin{align}
\frac{\mathrm d}{\mathrm dt} \|\delta \bfp(T-t)\|_{H^N}^2 &\le 2\rho\|\delta \bfp(T-t)\|_{H^N}^2- 2\theta\|\delta \bfp(T-t)\|_{V^N}^2\notag\\
&\quad+ \frac2 N\|\delta \bfA\|_{\clL(V^N,(V^N)')}\|\clE p_\sigma(T-t)\|_{V^N}\|\delta\bfp(T-t)\|_{V^N}\notag\\
&\quad+\frac2 N\| Q_\rme \delta \bfx(T-t)\|_{Y^N}\|Q_\rme\|_{\clL(H^N,Y^N)}\|\delta\bfp(T-t)\|_{H^N}\notag
\end{align}
and, by Young's inequality, 
\begin{align}
&\theta\|\delta \bfp(T-t)\|_{V^N}^2\!+\frac{\mathrm d}{\mathrm dt} \|\delta \bfp(T-t)\|_{H^N}^2 \le (\|Q_\rme\|_{\clL(H^N,Y^N)}^2\!+2\rho)\|\delta \bfp(T-t)\|_{H^N}^2+\Gamma(T-t),\notag
\end{align}
with $\Gamma(T-t)\coloneqq \frac1 {\theta N^2}\|\delta \bfA\|_{\clL(V^N,(V^N)')}^2\|\clE p_\sigma(T-t)\|_{V^N}^2+\frac1{ N^2}\| Q_\rme \delta \bfx(T-t)\|_{Y^N}^2$. 
Then, with~$2R\coloneqq \|Q_\rme\|_{\clL(H^N,Y^N)}^2+2\rho$, Gronwall's lemma and~\eqref{eq:optcondiffC} give
\begin{align}
 \|\delta \bfp(T-t)\|_{H^N}^2 &\le \max{(1,\rme^{2R T})}\left(\|\delta \bfp(T)\|_{H^N}^2 + \int_0^T \Gamma(T-s) \mathrm ds\right)\notag\\
&\hspace{-0em}\le\max{(1,\rme^{2R T})}\left(\frac{1}{N^2}\|P_\rme^*\|_{\clL(Z^N,H^N)}^2\|P_\rme\delta \bfx(T)\|_{Z^N}^2 + \int_0^T \Gamma(s) \mathrm ds\right),\notag
\end{align}
for $t \in [0,T]$. Further, using Lemma~\ref{lem:objdiff1}, we find~$\int_0^T \frac1N \|Q_\rme \delta \bfx(s)\|_{Y^N}^2 \mathrm ds +\frac{1}{N}\|P_\rme\delta \bfx(T)\|_{Z^N}^2\le 2\clJ(\delta\bfx, \delta u) \le 2 \mathfrak{C}_{1} \|\delta \bfA\|_{\clL(V^N,(V^N)')}^2$, which leads to
\begin{align}
N \|\delta \bfp(T-t)\|_{H^N}^2&\le \max{(1,\rme^{2R T})}\left(2\mathfrak{C}_{1} + \int_0^T \frac1{\theta N}\|\clE p_\sigma(s)\|_{V^N}^2 \mathrm ds\right)\|\delta \bfA\|_{\clL(V^N,(V^N)')}^2\notag\\
&= \max{(1,\rme^{2R T})}\left(2\mathfrak{C}_{1} +\theta^{-1}\|p_\sigma(s)\|_{V_T}^2 \right)\|\delta \bfA\|_{\clL(V^N,(V^N)')}^2.\notag
\end{align}

Finally,~\eqref{eq:optcondiffE} gives $\|\delta u_{\varSigma}(t)\|_{U} \le \|\bfB^\ast\|_{\clL(H^N,U)}\|\delta\bfp(t)\|_{H^N}$, for all~$t \in [0,T]$. Hence, the claim follows from~$\|\bfB^\ast\|_{\clL(H^N,U)} = \|\bfB\|_{\clL(U,H^N)} \le \sqrt{N} \|B\|_{\clL(U,H)}$.
\end{proof}

\subsection{Trajectories associated to~\eqref{eq:cllsys} and~\eqref{eq:OP1sig}}\label{sS:summary}
The following result is an immediate consequence of Theorems~\ref{thm:pointwise} and~\ref{thm:pointwisestate2}, and Corollaries~\ref{coro:controlpointwise1}, and~\ref{coro:controlpointwise2}.
\begin{corollary}\label{coro:optim-pair-comp}
Let $x_\circ \in V$, let $(x_\sigma,u_\sigma)$ be the minimizer of~\eqref{eq:OP1sig},
and let~$\clE x_{\varSigma,\sigma}$ the solution of~\eqref{eq:cllsys} with $u_{\varSigma,\sigma}(t) \coloneqq  K_{\varSigma}(t,\clE x_{\varSigma,\sigma}t))$. Then, we have the estimates
\begin{align}
\|\clE x_\sigma(t) - \clE x_{\varSigma,\sigma}(t)\|_{H^N} \le \fkC_3\|\delta \bfA\|_{\clL(\clD_\clA^N,H^N)} + \fkC_4 \|\delta \bfA\|_{\clL(V^N,(V^N)')}, \notag\\
\|u_\sigma(t) - u_{\varSigma,\sigma}(t)\|_{U} \le \fkC_5 \|\delta \bfA\|_{\clL(\clD_\clA^N,H^N)} + \fkC_6 \|\delta \bfA\|_{\clL(V^N,(V^N)')},\notag
\end{align}
for $t\in [0,T]$, where~$\delta \bfA=\bfA_\varSigma - \bfA_\sigma$ and
\begin{align}
\fkC_3 &= \mathfrak{C}_T \| \clE x_{\varSigma,\sigma} + \clE g \|_{L^2(0,T;\clD_\clA^N)},&\quad\!\!\!
\fkC_4 &= \max{\left(1,e^{(\frac12+\rho)T}\right)}\fkC_{T2},\notag\\
\fkC_5 &= \mathfrak{C}_{\mathrm{uni}2} \mathfrak{C}_T \| \clE x_{\varSigma,\sigma} + \clE g \|_{L^2(0,T;\clD_\clA^N)}&\quad\!\!\!
\fkC_6 &=\|B\|_{\clL(U,H)}\mathfrak{C}_{T3} \notag,
\end{align}
with~$\mathfrak{C}_T$ as in Theorem~\ref{thm:pointwise}; $\mathfrak{C}_{T2}$ as in Theorem~\ref{thm:pointwisestate2}; $\mathfrak{C}_{\mathrm{uni}2}$ as in Corollary~\ref{coro:controlpointwise1}; and~$\mathfrak{C}_{T3}$ as in Corollary~\ref{coro:controlpointwise2}, and~$\rho\in\bbR$ as in~\eqref{eq:aux1}.
\end{corollary}

\section{Numerical experiments}\label{S:numEx}

We present numerical experiments supporting our theoretical findings discussed in Section~\ref{S:optimCost}. Moreover, for a given target state $g$, and given a parameter ensemble~$\varSigma = (\sigma_i)_{i = 1}^N$ the performance of the feedback~$K_{\varSigma}$, as defined in~\eqref{eq:robFB}, is compared with the optimal feedback~$K_{\bar \sigma}$ for the ensemble average~$\bar \sigma \coloneqq \frac1N \sum_{i=1}^N \sigma_i$ of the parameters, that is we compare the closed-loop systems
\begin{align}
\label{eq:clsnum}
\dot{y}_\sigma(t) &= \clA_\sigma y_\sigma(t) + B K(t,y_\sigma(t)-g(t)), \qquad y_\sigma(0) = y_\circ,
\end{align}
for test parameters~$\sigma \in \fkS$, and for $K \in \{K_{\varSigma},K_{\bar \sigma}\}$, in~$[0,T]$ given by
\begin{align}
K_{\varSigma}(t, y_{\sigma}(t)-g(t)) &= -\bfB^* (\bfPi_{\varSigma}(T-t) \clE (y_{\sigma}(t)-g(t)) + \bfh(t)),\label{eq:robfbcont}\\
K_{\bar \sigma}(t,y_\sigma(t)-g(t)) &= -B^* (\Pi_{\bar \sigma}(T-t) (y_{\sigma}(t)-g(t)) + h(t)), \label{eq:averfbcont}
\end{align}
where~$\Pi_{\bar \sigma}$ solves $\dot{\Pi}_{\bar \sigma}(t) = \Pi_{\bar \sigma}(t) \clA_{\bar \sigma} + \clA_{\bar \sigma}^\ast \Pi_{\bar \sigma}(t) - \Pi_{\bar \sigma}(t) B B^\ast \Pi_{\bar \sigma}(t) + \frac1N Q^\ast Q$, $t \in [0,T]$ with $\Pi_{\bar \sigma}(0) = \Pi_\circ$, and $h$ solves $-\dot{h}(t) = \left(A_{\bar \sigma}^\ast - \Pi_{\bar \sigma}(T-t)B B^\ast\right) h(t) + \Pi_{\bar \sigma}(T-t) (\clA_{\bar \sigma} g(t) - \dot{g}(t))$, $t \in [0,T)$, with $h(T) = 0$ (cf.~Section~\ref{sS:optim-ext}; note also that~$\Pi_{\bar{\sigma}} = \Pi_{\{\bar{\sigma}\}}$).

Given a parameter~$\sigma \in \fkS$, we denote the solution of~\eqref{eq:clsnum} with~$K = K_{\varSigma}$ by~$y_{\varSigma,\sigma}$, and the solution of~\eqref{eq:clsnum} with~$K = K_{\bar \sigma}$ by~$y_{\bar \sigma,\sigma}$.

\subsection{Oscillator}
Let us consider the differential equation
\begin{align}\label{eq:oscillator2nd}
\ddot{\theta}(t) &= - \theta(t) - \sigma \dot{\theta}(t) + u(t),\quad t \in [0,T],\\
\theta(0) &= \theta_\circ, \quad \dot{\theta}(0) = \theta_{\circ,1}.
\end{align}
Thus, the damping parameter~$\sigma$ is allowed to be uncertain. We consider an ensemble~$\varSigma = (\sigma_i)_{i=1}^N$ of possible values of~$\sigma$, and write the second order equation~\eqref{eq:oscillator2nd}, for each~$\sigma_i$, as
\begin{align}\label{eq:oscillator1st}
\begin{split}
\dot{y}_{\sigma_i}(t) &= \clA_{\sigma_i} x_{\sigma_i}(t) + B u(t), \quad t \in [0,T],\quad  1\le i \le N,\\
y_{\sigma_i}(0) &= y_\circ, \quad \quad \quad\quad \quad \quad \quad \quad \quad\quad \quad \quad\quad\, 1\le i \le N,
\end{split}
\end{align}
with~$\clA_{\sigma_i} = \begin{bmatrix} 0 &1 \\ -1 &-\sigma_i \end{bmatrix}$, $B = \begin{bmatrix} 0 \\ 1 \end{bmatrix}$, initial condition $y_\circ = \begin{bmatrix} \theta_\circ\\ \theta_{\circ,1} \end{bmatrix} \in \mathbb{R}^2$, and corresponding states~$y_{\sigma_i} = \begin{bmatrix} \theta &\dot{\theta}\end{bmatrix}^\top =  \begin{bmatrix} \theta_{\sigma_i}& \dot{\theta}_{\sigma_i}\end{bmatrix}^\top$.
The target function $g$, is chosen to solve~\eqref{eq:oscillator1st} with $\sigma = 1$ and initial condition $y_{\circ} = \begin{bmatrix} 1 & 0\end{bmatrix}^\top$. Further, we set~$T=5$, $Q = \begin{bmatrix} \sqrt{10} & 0 \end{bmatrix}^\top$, and $P = \mathbf{1}_{\bbR^2}$, where~$\mathbf{1}_{\bbR^2}$ denotes the identity matrix in~$\bbR^2$. 

The parameter ensembles will be described as
\begin{align*}
\varSigma^{R+1}_\ell \coloneqq \Big\{\Big(-1 + \frac{2r}{R}\Big)\ell \mid 0\le r \le R\Big\},
\end{align*}
where~$R+1$ denotes the cardinality of the ensemble, i.e., the number of parameters~$\sigma_r\in\varSigma^{R+1}_\ell$, $0\le r\le R$, and~$\ell >0$ determines the range of the parameter set, and hence resembles the level of uncertaintiy in the problem. The feedback~$K_{\bar{\sigma}}$ in~\eqref{eq:averfbcont} is not affected by changes in~$\ell$ or~$R \in \bbN$, since~$\bar{\sigma} = 0$.

\begin{figure}[tb]
\centering
  \includegraphics[width=\linewidth]{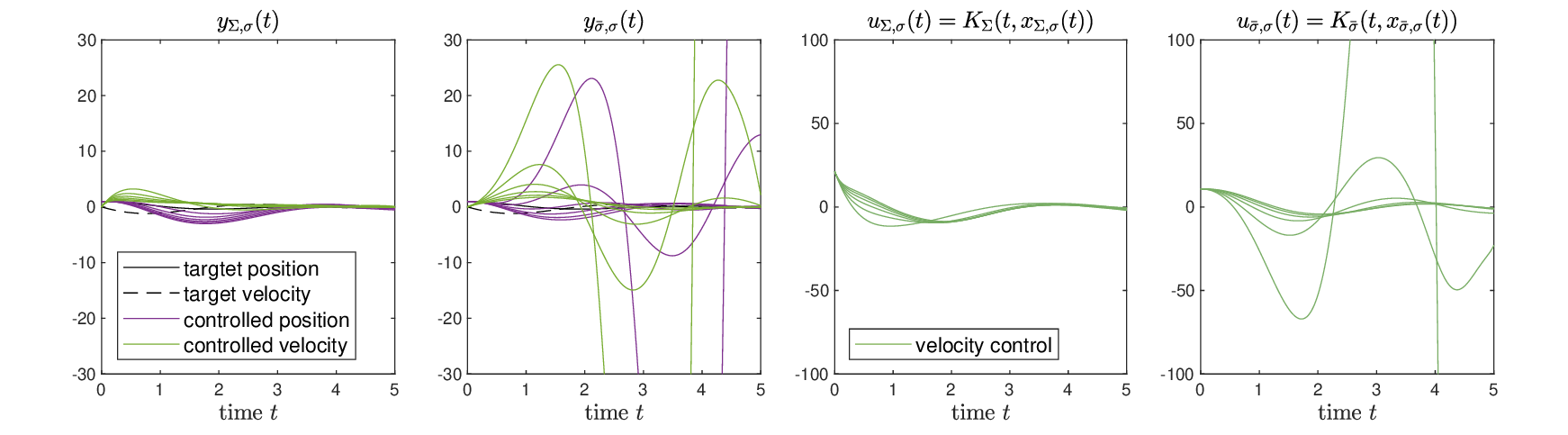}
  \caption{The feedback controls are computed using~$\varSigma_{2}^{5}$ and tested on~$\varSigma_{4}^{6}$. Left: state trajectories corresponding to the feedback control~\eqref{eq:robfbcont} (left) and the feedback control~\eqref{eq:averfbcont} (middle left). Right: feedback control~\eqref{eq:robfbcont} (middle right) and feedback control~\eqref{eq:averfbcont} (right). }
\label{fig:osc1}
\end{figure}

In Figure~\ref{fig:osc1} the feedback control~\eqref{eq:robfbcont} is compared to the feedback control~\eqref{eq:averfbcont} along with the corresponding closed-loop state trajectories~$y_{\varSigma,\sigma}$ and~$y_{\bar \sigma,\sigma}$, respectively. Here, the feedbacks~~\eqref{eq:robfbcont} and~\eqref{eq:averfbcont} are constructed based on the training ensemble~$\varSigma_2^5$, and then tested in the systems with parameters~$\sigma \in \varSigma_{4}^{6}$, i.e.,~$6$ test trajectories for each component (position and velocity) are displayed. It is observed that the feedback control~\eqref{eq:robfbcont} is much more robust with respect to parameter variations than~\eqref{eq:averfbcont}: the feedback~\eqref{eq:averfbcont} leads to worthless controls, which fail to track the target for the two most unstable test parameters~$\sigma=-4$ and~$\sigma = -2.4$, whereas the feedback~\eqref{eq:robfbcont} still steers the respective states close to the target.

\begin{figure}[tb]
\centering
  \includegraphics[width=\linewidth]{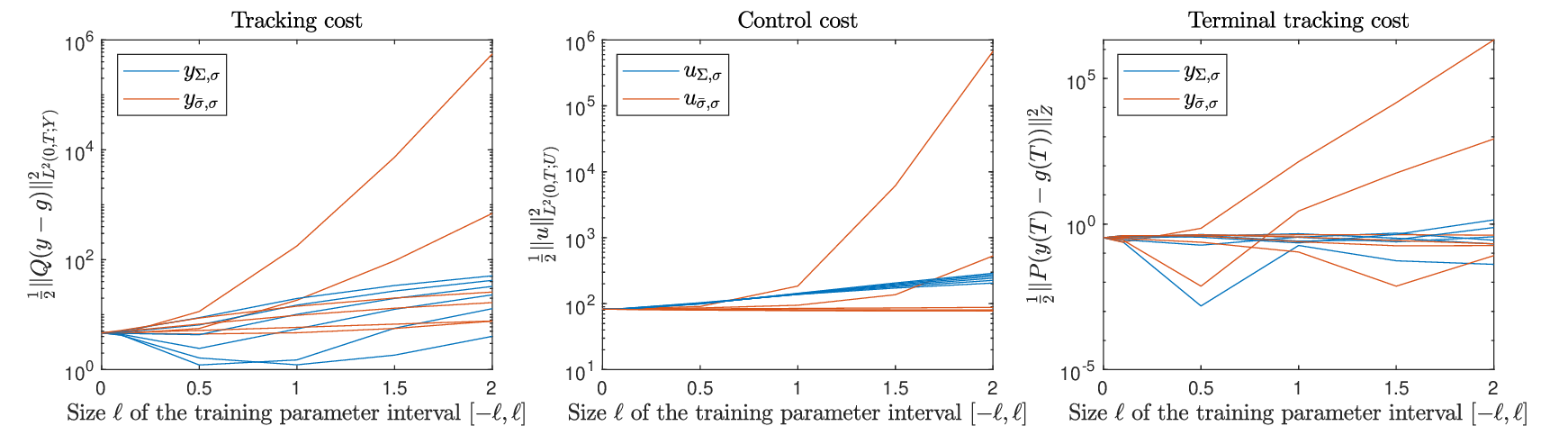}
  \caption{Increasing training parameter interval~$\varSigma_\ell^5$, and increasing test parameter interval~$\varSigma^{6}_{2\ell}$ for~$\ell \in \{0,\frac{1}{10},\frac{1}{2},1,\frac{3}{2},2\}$. Left: tracking cost~$\frac12 \|Q(y_\sigma - g)\|^2_{L^2(0,T;Y)}$. Middle: feedback control cost~$\frac12 \|u\|^2_{L^2(0,T;U)}$. Right: terminal tracking cost~$\frac12 \|P(y_\sigma - g)\|^2_{Z}$.}
  \label{fig:osc2}
\end{figure}

\begin{figure}[tb]
\centering
  \includegraphics[width=\linewidth]{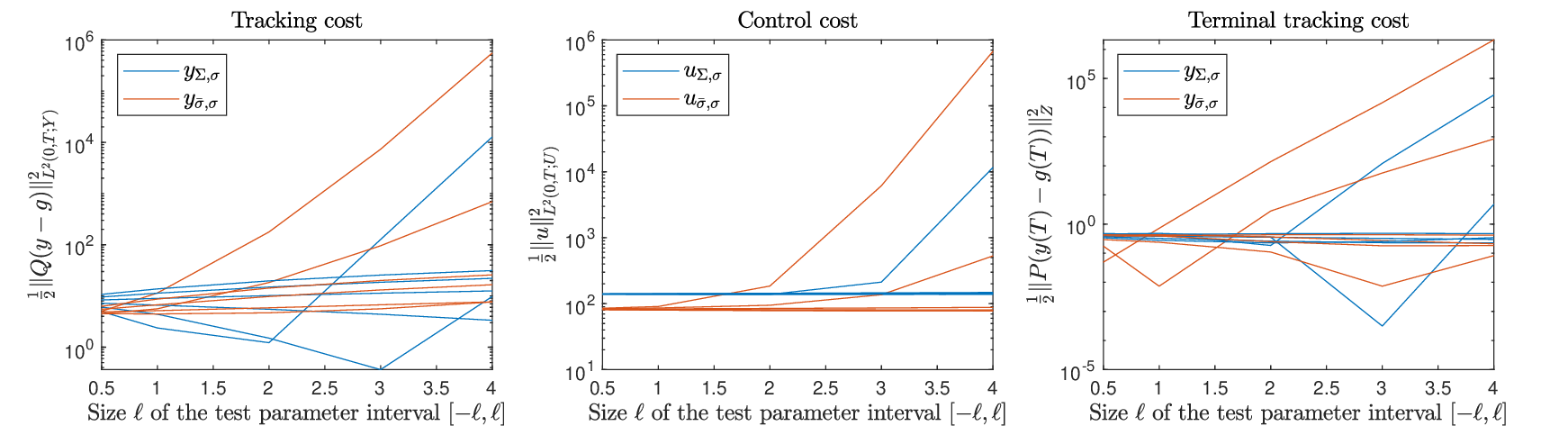}
  \caption{Fixed training parameter interval~$\varSigma_1^5$, and increasing test parameter interval~$\varSigma^{6}_\ell$ for~$\ell \in \{\frac{1}{2},1,2,3,4\}$. Left: tracking cost~$\frac12 \|Q(y_\sigma - g)\|^2_{L^2(0,T;Y)}$. Middle: feedback control cost~$\frac12 \|u\|^2_{L^2(0,T;U)}$. Right: terminal tracking cost~$\frac12 \|P(y_\sigma - g)\|^2_{Z}$.}
  \label{fig:osc4}
\end{figure}
The superior robustness against parameter variations of~\eqref{eq:robfbcont} compared to~\eqref{eq:averfbcont} is also reflected in the associated costs, which are displayed for different levels of uncertainty $\ell \in \{0,\frac{1}{10},\frac{1}{2},1,\frac{3}{2},2\}$ in Figure~\ref{fig:osc2}. Here, the feedbacks~~\eqref{eq:robfbcont} and~\eqref{eq:averfbcont} are constructed based on the training ensembles~$\varSigma_\ell^5$, and then tested in the systems with parameters~$\sigma \in \varSigma_{2\ell}^{6}$. 
It is observed that, with increasing level of uncertainty~$\ell$, the feedback~\eqref{eq:averfbcont} leads to rapidly increasing tracking cost for the two most unstable test parameters~$\sigma=-4$ and~$\sigma = -2.4$, whereas the tracking cost associated with~\eqref{eq:robfbcont} grows much slower. For~$\ell = 2$, the largest tracking costs in the parameter test set~$\varSigma_{2\ell}^{6}$ are~$\frac12 \|Q(y_{\bar{\sigma},-4} - g)\|_{L^2(0,T;Y)}^2 = 552870$ and~$\frac12 \|Q(y_{\varSigma_\ell^5,4} - g)\|_{L^2(0,T;Y)}^2 \approx 51.0$, and the terminal tracking costs are~$\frac12 \|Q(y_{\bar{\sigma},-4} - g)\|_{Z}^2 = 2115000$ and~$\frac12 \|Q(y_{\varSigma_\ell^5,4} - g)\|_{L^2(0,T;Z)}^2 \approx 1.4$.

For less extreme test parameters, which result in more stable systems, the tracking performance of both feedbacks is similar, while the robust feedback in this case comes at higher control cost, see Figure~\ref{fig:osc2}. However, for the two most unstable test parameters~$\sigma = -4$ and~$\sigma = -2.4$, the feedback~\eqref{eq:robfbcont} also leads to smaller control cost than the feedback~\eqref{eq:averfbcont}.

In Figure~\ref{fig:osc4}, the feedbacks~\eqref{eq:robfbcont} and~\eqref{eq:averfbcont} are compared for a fixed training parameter interval~$\varSigma_1^5$ and increasing test parameter interval~$\varSigma_\ell^6$ for~$\ell \in \{0.5,1,2,3,4\}$. A similar relationship is observed: for larger levels of uncertainty~$\ell$, the cost associated with the feedback~\eqref{eq:averfbcont} (in red) is much larger for unstable systems than the cost associated with the feedback~\eqref{eq:robfbcont} (in blue). For stable systems, the tracking performance of the both feedbacks is again similar. Overall, for this example the robustness of the feedback~\eqref{eq:robfbcont} comes at the possible expense of higher control cost.

Finally, in accordance with Corollary~\ref{coro:smallrob}, the costs converge  as the difference of the test parameters tend to zero, i.e., as~$\ell$ tends to $0$, see  Figure~\ref{fig:osc2}.

\subsection{Convection-diffusion-reaction equation}

Let us consider the parameterized convection-diffusion-reaction equation under Neumann boundary conditions as follows
\begin{align*}
\dot{y}_{\sigma} - \nabla \cdot (a_\sigma \nabla y_\sigma) + c y_\sigma + \nabla \cdot(b y_\sigma) &= \sum_{i = 1}^{N_{\mathrm a}} u_i \mathbf{1}_{O_i} &&(t,s) \in (0,T] \times D,\\
\frac{\partial y_\sigma}{\partial \mathrm n} &= 0 &&(t,s) \in [0,T] \times \partial D,\\
y_\sigma &= y_\circ &&(t,s) \in \{t=0\} \times D,
\end{align*}
where~$T=5$,~$D = (0,1)$ with boundary~$\partial D = \{0,1\}$ and the functions~$\mathbf{1}_{O_i}$ represent the support of the actuators, which are modelled as the characteristic functions related to open sets~$O_i \subset D$ for~$1\le i \le N_a$. It is assumed that the reaction coefficient~$c$ and the convection coefficient~$b$ are given constants, and that the parameter~$\sigma = (\sigma_1,\ldots,\sigma_{N_s}) \in \bbR^{N_s}$ enters the diffusion, that is
\begin{align}\label{eq:lognormal}
a_\sigma(s) = \bar{a}(s) \exp\bigg(\sum_{j=1}^{N_s} \sigma_j \psi_j(s)\bigg),
\end{align}
for~$\bar{a} \in \mathcal{C}^0(\overline{D})$ and $\psi_j \in L^\infty(D)$ for all $1 \le j \le N_s$. Assuming that~$\bar{a}(s)>0$~$\forall s\in D$, it follows that there exist~$a_{\sigma,\min}$ and~$a_{\sigma,\max}$ depending on~$\sigma$ such that
\begin{align*}
0<a_{\sigma,\min} \le a_\sigma(s) \le a_{\sigma,\max} < \infty \quad \mbox{for all } s \in D \mbox{ and } \sigma \in \bbR^{N_s}.
\end{align*}
The representation~\eqref{eq:lognormal} is called a lognormal parameterization, if the parameters~$\sigma = (\sigma_j)_{j=1}^{N_s}$ are independently and identically distributed (i.i.d.) standard normal random variables, that is~$(\sigma_j)_{j=1}^{N_s} \sim \bbP \coloneqq \bigotimes_{j=1}^{N_s} \mathcal{N}(0,1)$, see, e.g.,~\cite{BabsukaNobileTempone}. Parameterizations of this form have origins in Karhunen--Lo\`eve expansions of lognormal random fields, see, e.g.,~\cite{schwab_gittelson_2011}.

\begin{figure}[tb]
\centering
\begin{minipage}[t]{.48\textwidth}
  \includegraphics[width=\linewidth]{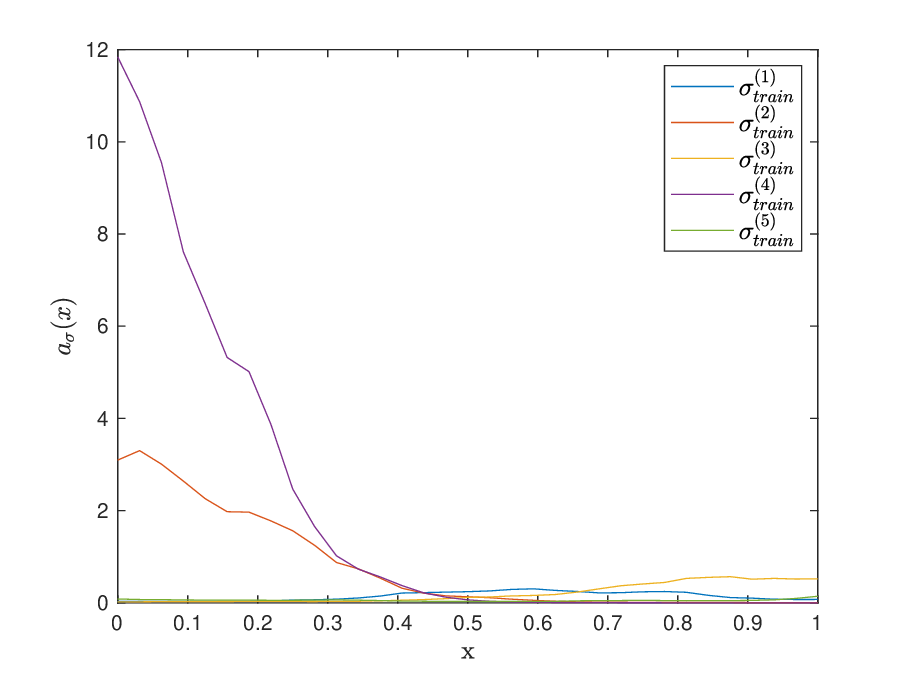}
\end{minipage}
\begin{minipage}[t]{.48\textwidth}
  \includegraphics[width=\linewidth]{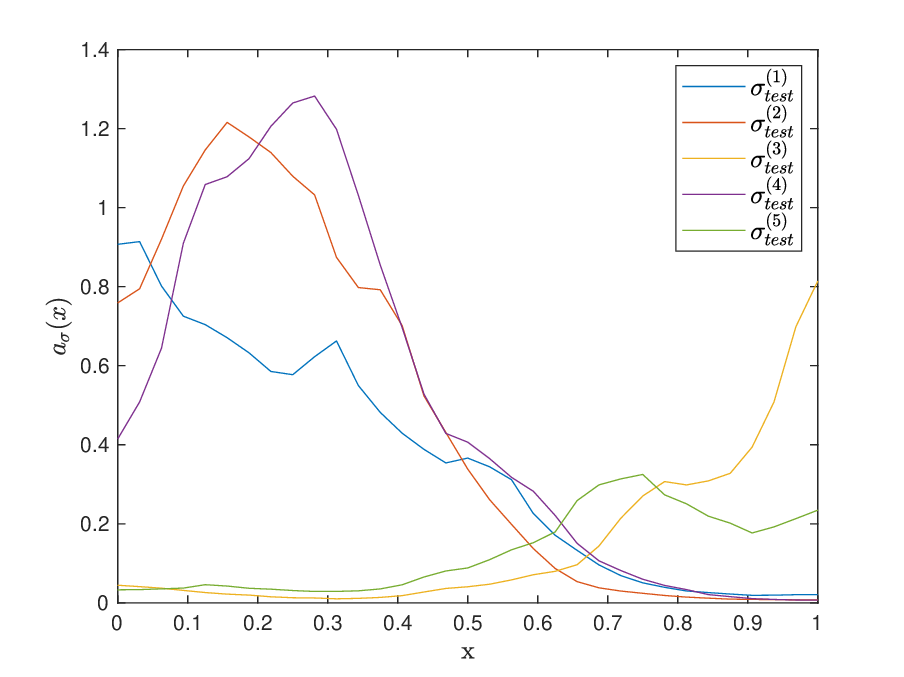}
\end{minipage}
  \caption{Five realizations of the random diffusion coefficient~\eqref{eq:lognormal}. Left: training set. Right: test set.}
  \label{fig:osc11}
\end{figure}
\begin{figure}[tb]
\centering
  \includegraphics[width=\linewidth]{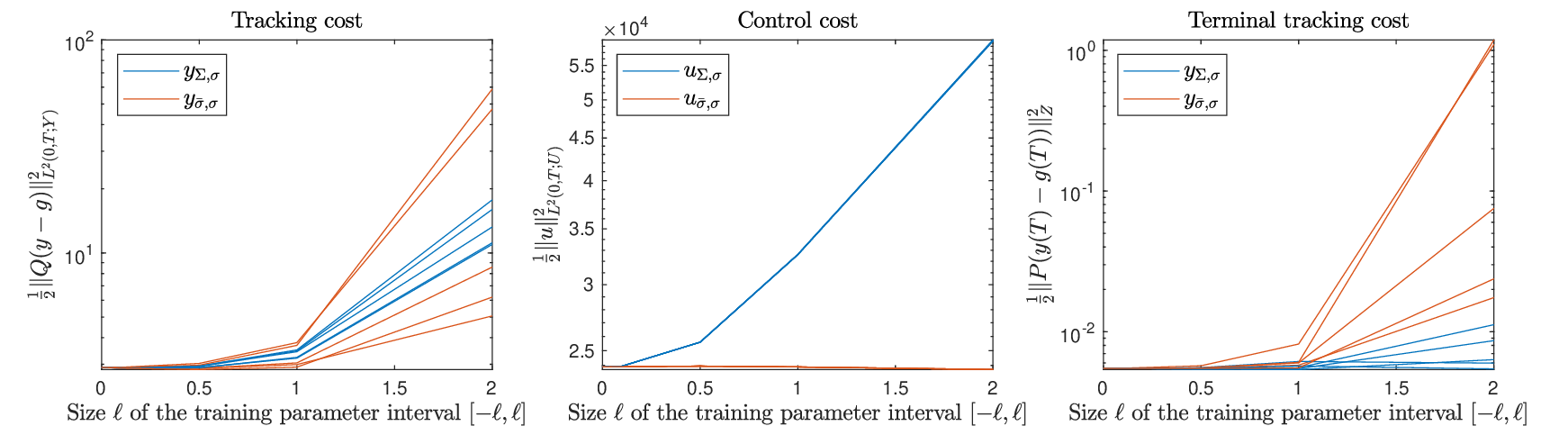}
\caption{Cost for a set of $N=5$ test parameters. Left: tracking cost~$\frac12 \|Q(y_\sigma - g)\|^2_{L^2(0,T;Y)}$. Middle: feedback control cost~$\frac12 \|u\|^2_{L^2(0,T;U)}$. Right: terminal tracking cost~$\frac12 \|P(y_\sigma - g)\|^2_{Z}$.}
\label{fig:osc8}
\end{figure}
\begin{figure}[tb]
\centering
\begin{minipage}[t]{1.\textwidth}
  \includegraphics[width=\linewidth]{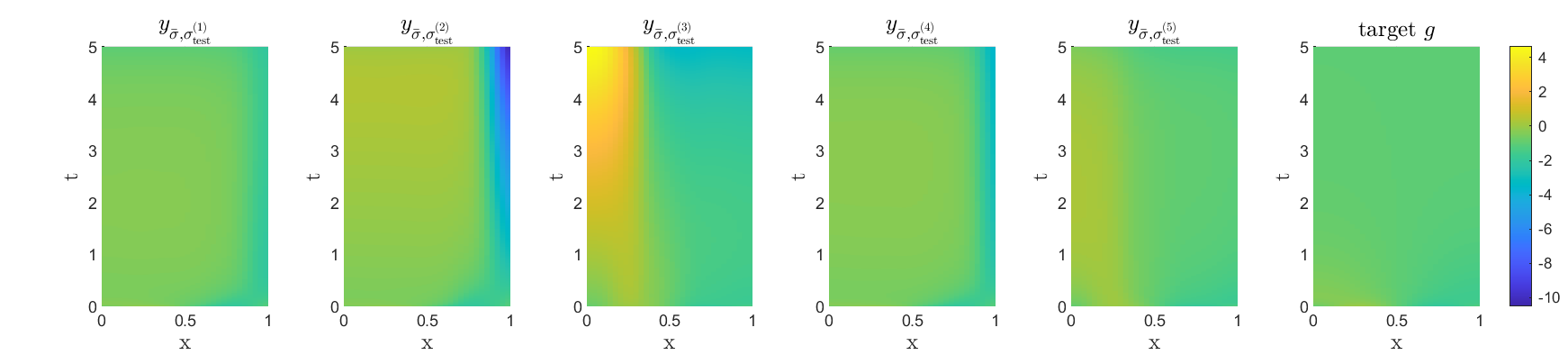}
\end{minipage}\\
\begin{minipage}[t]{1.\textwidth}
  \includegraphics[width=\linewidth]{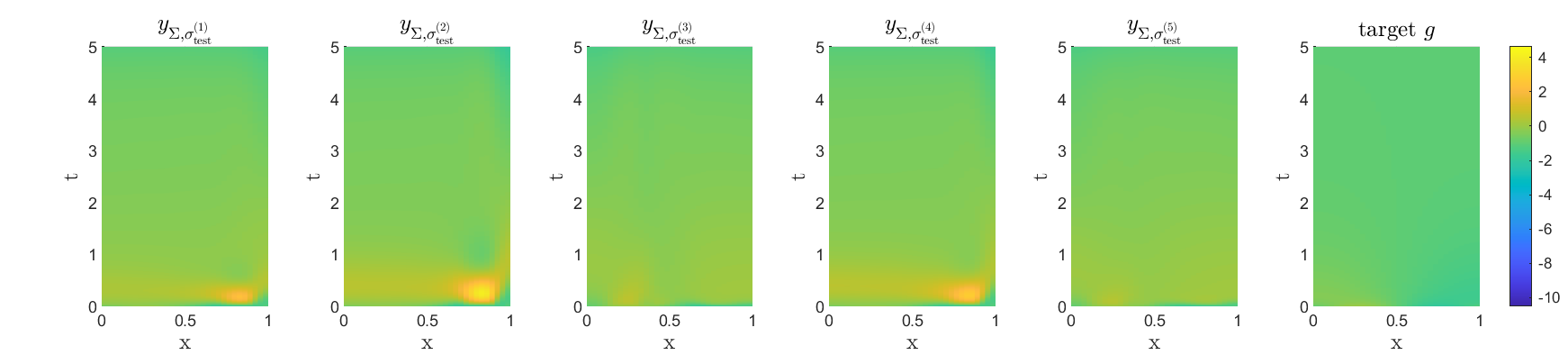}
\end{minipage}
  \caption{Without convection $b=0$. Five realizations of state trajectories corresponding to the feedback control~\eqref{eq:robfbcont} (bottom) and the feedback control~\eqref{eq:averfbcont} (top).}
  \label{fig:osc9}
\end{figure}
\begin{figure}[tb]
\centering
  \includegraphics[width=\linewidth]{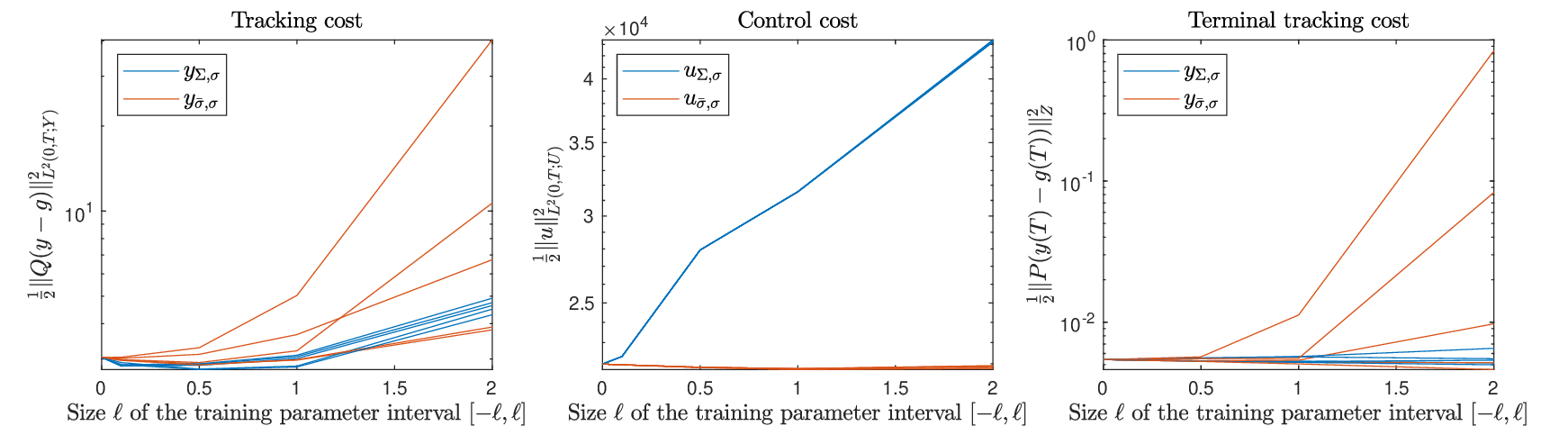}
\caption{Cost for a set of $N=5$ test parameters. Left: tracking cost~$\frac12 \|Q(y_\sigma - g)\|^2_{L^2(0,T;Y)}$. Middle: feedback control cost~$\frac12 \|u\|^2_{L^2(0,T;U)}$. Right: terminal tracking cost~$\frac12 \|P(y_\sigma - g)\|^2_{Z}$.}
\label{fig:osc10}
\end{figure}
\begin{figure}[tb]
\centering
\begin{minipage}[t]{1.\textwidth}
  \includegraphics[width=\linewidth]{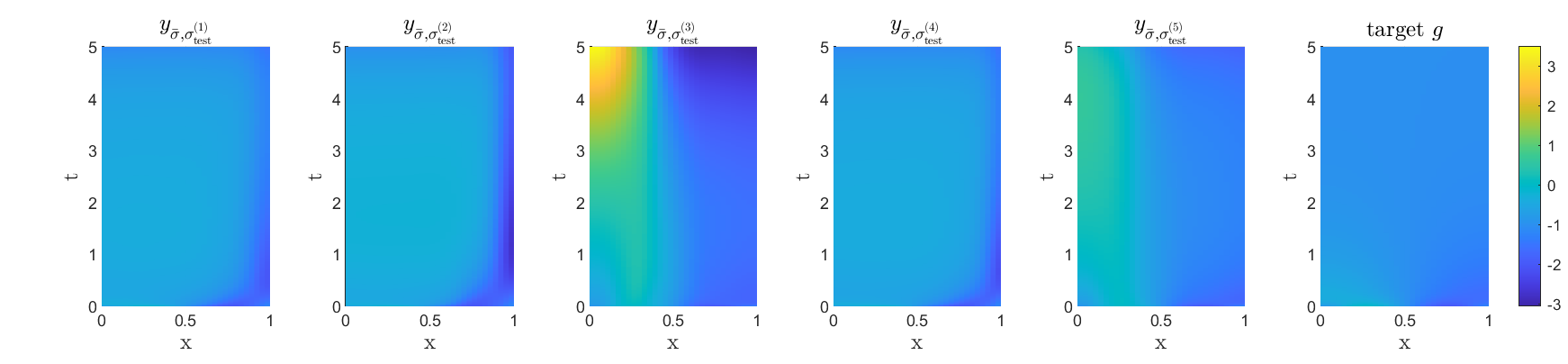}
\end{minipage}\\
\begin{minipage}[t]{1.\textwidth}
  \includegraphics[width=\linewidth]{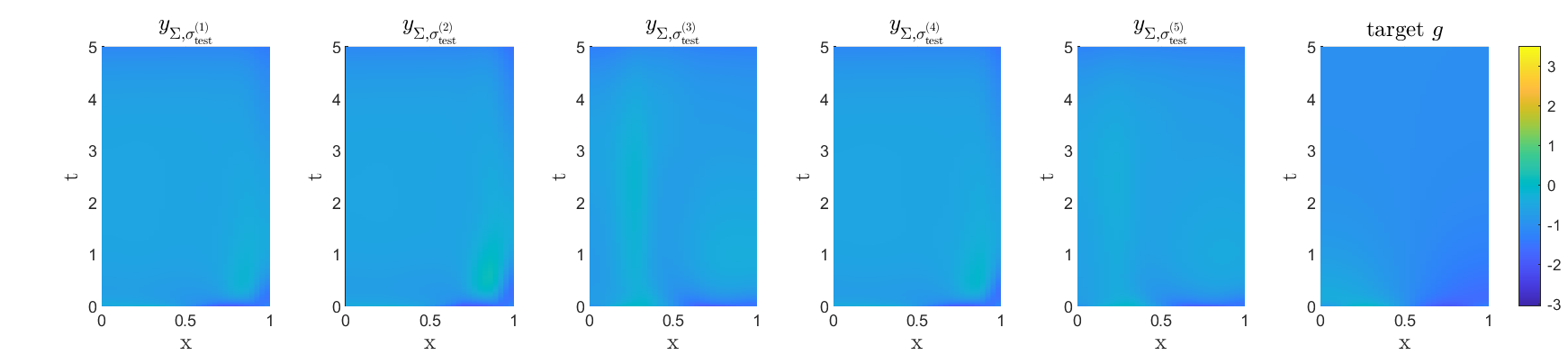}
\end{minipage}
  \caption{With convection $b=0.1$. Five realizations of state trajectories corresponding to the feedback control~\eqref{eq:robfbcont} (bottom) and the feedback control~\eqref{eq:averfbcont} (top).}
  \label{fig:osc11}
\end{figure}

In the numerical experiments~$N_a = 3$ actuators are used as~$O_1 = [0.1, 0.3]$, $O_2 = [0.4, 0.6]$, and $O_3 = [0.7, 0.9]$. In~\eqref{eq:lognormal} the mean field is set to~$\bar{a}=0.1$ and it is assumed that the diffusion depends on~$N_s = 100$ realizations of i.i.d.~standard normal random variables, and parametric basis functions~$\psi_{2j}(x) = (2j)^{-\nu} \sin(j\pi x)$ and~$\psi_{2j-1}(x) = (2j-1)^{-\nu} \cos(j\pi x)$ with~$\nu = \frac{3}{2}$, cf.~\cite{gantner}. Further, a constant reaction coefficient~$c =-1$ is assumed, and we set~$Q = \sqrt{10}\cdot \mathcal{P}_F$, where~$\mathcal{P}_F$ is the orthogonal projection in~$H$ onto~$\mathrm{span}(F)$, where~$F = \begin{bmatrix} \mathbf{1} & \cos(\pi x) & \cos(2\pi x)\end{bmatrix}^\top$, and $P = \mathbf{1}_H$ in~\eqref{eq:obj}, where~$\mathbf{1}_H$ denotes the identity operator in~$H$, as well as the initial condition~$y_\circ(x) = \sin(2\pi x)-1$. The target~$g$ solves the heat equation~$\dot{g} = 0.1 \Delta g$ with the same boundary and initial data.

To construct the feedbacks,~$5$ training vectors, each containing~$N_s = 100$ realizations of i.i.d.~standard normal random variables, are drawn. In order to investigate different variance levels in Figures~\ref{fig:osc8} and~\ref{fig:osc10}, the training vectors are multiplied by a scalar~$\ell \in \{0,\frac{1}{10},\frac{1}{2},1,2\}$. The feedbacks are then tested with~$5$ different test parameters, each of which consists of~$N_s = 100$ realizations of i.i.d.~standard normal random variables.

Results without convection are displayed in Figure~\ref{fig:osc8} and Figure~\ref{fig:osc9}, and results with convection $b=0.1$ are displayed in Figure~\ref{fig:osc10} and Figure~\ref{fig:osc11}. In the case without convection the feedback~\eqref{eq:robfbcont} (in blue) has smaller terminal tracking costs than the feedback~\eqref{eq:averfbcont} (in red) for all tested diffusion coefficients, see Figure~\ref{fig:osc8}. In addition, it has smaller tracking cost for some of the tested diffusion coefficients. The more robust tracking performance comes at the expense of higher control cost for~\eqref{eq:robfbcont}.

Similarly, in the case with convection~$b=0.1$, the feedback~\eqref{eq:robfbcont} tracks the target better than the feedback~\eqref{eq:averfbcont}: while~\eqref{eq:averfbcont} clearly fails to track the target for~$\sigma_{\mathrm{test}}^{(2)}$, $\sigma_{\mathrm{test}}^{(3)}$, and~$\sigma_{\mathrm{test}}^{(5)}$, the feedback~\eqref{eq:robfbcont} tracks the target much better for the tested realizations of the diffusion coefficient~\eqref{eq:lognormal}, see Figure~\ref{fig:osc11}. The improved tracking performance of~\eqref{eq:robfbcont} is reflected in the associated costs in Figure~\ref{fig:osc10} and again comes at the expense of higher control cost. The control costs are insensitive to changes of the test parameters, such that no difference between the cost for different test parameters can be seen in Figure~\ref{fig:osc10}. The same phenomenon is observed in Figure~\ref{fig:osc8}.

In summary, in both cases the feedback~\eqref{eq:robfbcont} is more robust against variations in the diffusion coefficient at the cost of higher control costs.

Finally, in accordance with Corollary~\ref{coro:smallrob}, the costs converge as the difference of the test parameters tend to zero, i.e., as~$\ell$ tends to $0$, see Figures~\ref{fig:osc8} and~\ref{fig:osc10}.


\bigskip\noindent
{\bf Aknowlegments.}
S. Rodrigues gratefully acknowledges partial support from
the State of Upper Austria and Austrian Science
Fund (FWF): P 33432-NBL.





\bibliographystyle{plainurl}
\bibliography{robust_tracking}


%
%
%
\end{document}